\documentclass[11pt, twoside]{article}

\usepackage{graphicx}
\usepackage[caption=false]{subfig}
\usepackage{amsmath}
\usepackage{amssymb,amsfonts}
\usepackage{amsthm}
\usepackage{bm}
\usepackage{mathrsfs}
\usepackage{amssymb}
\usepackage{multirow}

\usepackage[margin=1in]{geometry}

\usepackage{algorithm}
\usepackage{algorithmic}

\numberwithin{equation}{section}
\usepackage{multirow}
\usepackage{makecell}
\usepackage{booktabs}
\theoremstyle{definition}
\newtheorem{theorem}{Theorem}

\newtheorem{lemma}{Lemma}
\newtheorem{corollary}{Corollary}
\newtheorem{definition}{Definition}
\newtheorem{remark}{Remark}
\newtheorem{example}{Example}
\newtheorem{assumption}{Assumption}

\usepackage{cite}
\usepackage{hyperref}
\usepackage[nameinlink]{cleveref}
\newcommand{\vertiii}[1]{{\left\vert\kern-0.25ex\left\vert\kern-0.25ex\left\vert #1 
		\right\vert\kern-0.25ex\right\vert\kern-0.25ex\right\vert}}

\usepackage{fancyhdr}
\begin{document}
	
	\title{HDG and CG methods  for the Indefinite Time-Harmonic Maxwell's Equations under minimal regularity}

\author{Gang Chen%
	\thanks{ College of  Mathematics, Sichuan University, Chengdu 610064, China (\mbox{cglwdm@uestc.edu.cn}).}
	\and
	Peter Monk%
	\thanks{Department of Mathematics Science, University of Delaware, Newark, DE, USA (\mbox{monk@udel.edu}).}
	\and
	Yangwen Zhang%
	\thanks{Department of Mathematics Science, University of Delaware, Newark, DE, USA (\mbox{ywzhangf@udel.edu}).}
}

\date{\today}

\maketitle

\begin{abstract}
 We propose to use a hybridizable discontinuous Galerkin (HDG) method combined with the continuous Galerkin (CG) method to approximate  Maxwell's equations. We make two contributions in this paper. First, even though there are many papers using HDG methods to approximate Maxwell's equations, to our knowledge they all assume that the coefficients are smooth (or constant). Here, we derive optimal convergence estimates for our HGD-CG approximation when the electromagnetic coefficients are  {\em piecewise} smooth. This requires new techniques of analysis. Second, we use  CG elements to approximate the Lagrange multiplier used to enforce the divergence condition and  we obtain  a discrete system in which we can decouple the discrete the Lagrange multiplier. Because we are using a continuous Lagrange multiplier space, the
number of degrees of freedom devoted to this are less than for other
HDG methods. We present numerical experiments to confirm our theoretical results.
\end{abstract}

\section{Introduction}\label{}
Maxwell's equations govern the propagation of electromagnetic waves and  have wide applications in science and technology; such as aerospace industry, telecommunication, medicine, and biology. Hence, a large number of computational techniques have been developed for solving Maxwell's equations, including finite difference methods, integral equation methods and finite element methods.  Amongst these techniques, the finite element method is a popular method for the solution of time-harmonic electromagnetic problems due to its
ability to handle complex geometries and inhomogeneous materials.

Let $\Omega\subset \mathbb R^3$ be a  {simply connected} Lipschitz polyhedral domain with  {connected} boundary $\partial \Omega$, we consider the {\em indefinite} time-harmonic Maxwell equations  with a perfectly conducting boundary: find $(\bm u, p)$ that satisfies
\begin{subequations}\label{Maxwell_PDE_ori}
	\begin{align}
	\bm\nabla  \times (\mu_r^{-1}  \bm\nabla\times\bm u)-\kappa^2\epsilon_r\bm u + \overline{\epsilon_r}\nabla p&=\bm f &\text{in }\Omega,\label{Maxwell_PDE_ori1}\\
	\bm \nabla\cdot(\epsilon_r\bm u)&= {\rho} &\text{in }\Omega,\label{Maxwell_PDE_ori2}\\
	\bm n\times\bm u&=\bm 0&\text{on }\partial\Omega,\label{Maxwell_PDE_ori3}\\
	p&=0&\text{on }\partial\Omega.\label{Maxwell_PDE_ori4}
	\end{align}
\end{subequations}
Here $\mu_r$ and $\epsilon_r$ are the relative magnetic permeability and the relative electric permittivity,  which may be complex valued  {(and the overbar denotes complex conjugation)}.  In addition
$\bm f =ik\epsilon_0\bm j$,  where $\bm j$ is the given current density and $\epsilon_0$ is the permittivity of vacuum, and $\kappa>0$ is the wave number.  {The function $\rho$ denotes the charge density. We note that the topological assumption can be relaxed~\cite{Monk_Maxwell_Book_2003}, but we choose the simplest setting here.}

The Lagrange multiplier $p$ is present to stabilize problem by allowing explicit imposition of the divergence constraint \eqref{Maxwell_PDE_ori2}, and is important to stabilize low frequency (small $\kappa$) problems~\cite{demko98} including the special case $\kappa=0$ which corresponds to an electrostatic field.  In electromagnetism, the function $\bm f$ and $\rho$ are not independent, because conservation of charge requires that $\nabla\cdot\bm f+ \kappa^2 \rho=0$, so that $p=0$.  To simplify the presentation,  shall assume that $\rho=0$ in the remainder of the paper.

{Among finite element methods,  curl-conforming edge elements (i.e. in $\bm{H}(\text{curl};\Omega)$)} have been widely studied, see for example \cite{Nedelec_mixed_NM_1986,Nedelec_mixed_NM_1980,Hiptmair_electromagnetism_Acta_2002,Monk_Maxwell_Book_2003,Monk_Maxwell_NM_1992,Zhong_Maxwell_JCM_2009}.  {These are often referred to as edge elements,} and 
are known to eliminate the problem of spurious modes which may arise when standard finite elements are used to
discretize Maxwell’s equations \cite{Bossavit_IEEE_1990}.  Low-order edge elements are often used for problems in electromagnetics because they can be easily implemented, however, the use of low-order edge
elements often leads to a discrete linear system with a large number of unknowns, especially for electromagnetic problems at high frequencies  {(large $\kappa$)}. As a result, high-order edge elements have been developed \cite{Ainsworth_CMAME_2001,Ainsworth_IJNME_2003,demko98} and shown to be more effective than low-order edge elements. However, high-order edge elements introduce
extra degrees of freedom in the interior of the elements which increase dramatically with the order of approximation. 
Typically, these interior degrees of freedom can be eliminated by using a procedure known as static condensation \cite{Ledger_ACME_2005}. However, the implementation of high order edge element methods is complicated. Hence,  non-conforming methods provide an  interesting   alternative for this kind of problem and may also be attractive for nonlinear problems.

{Interior Penalty }Discontinuous Galerkin (DG) methods  have also been used to approximate the solution of the Maxwell's equations for some time. The first DG method for solving
Maxwell's equations with high frequency was analyzed in \cite{Monk_Maxwell_CMAME_2002} and much improved in \cite{Houston_Maxwell_NM_2005}.  In \cite{Hesthaven_JCP_2002,Hesthaven_PTRSLAMP_2004}, the local discontinuous Galerkin (LDG) method with high-order nodal elements is used to solve Maxwell's equations.   {From these studies we see that} DG methods have several distinct advantages including their capabilities to handle complex geometries, to provide high-order accurate solutions, to perform $hp$ adaptivity, and to retain excellent scalability. However, many existing DG methods are known to be computationally expensive because they have too many degrees of freedom due to nodal duplication on element boundaries.

{In part to improve computational efficiency,} Hybridizable discontinuous Galerkin (HDG) methods were proposed by Cockburn et al. in \cite{Cockburn_Gopalakrishnan_Lazarov_Unify_SINUM_2009}. HDG methods are based on a mixed formulation and utilize a numerical flux and a numerical trace to approximate the flux and the trace of the solution. The  {volume based} approximate flux and solution variables can be eliminated element-by-element  {condensation}. This process leads to a global equation for the approximate boundary traces only. As a result, HDG methods have significantly less globally coupled unknowns, a smaller memory requirement, and computational cost compared to other DG methods. Furthermore, HDG methods have been applied to Maxwell's equations in \cite{Cockburn_maxwell_HDG_JCP_2011} but without an error analysis. Later on, an error analysis was provided 
in \cite{Chen_maxwell_HDG_2017,Chen_Maxwell_HDG_CMAME_2018,Du_Arxiv_2019}  for  zero frequency and  in \cite{Feng_Maxwell_CMAM_2016,Lu_hp_HDG_Maxwell_Math_Comp_2017} for impedance boundary conditions and high wave number. In a very recent paper~\cite{ChenMonkZhang1}, we used the  concept of an $M$-decomposition, which was proposed by Cockburn et al in \cite{Cockburn_M_decomposition_Part1_Math_Comp_2017} for elliptic PDEs to analyze HDG schemes for Maxwell's equations in two dimensions.  This analysis provides conditions on the HDG spaces need to obtain optimal convergence, and superconvergence of some variables.  The extension of this approach to 3D is challenging, and remains to be done. It is  worthwhile to mention  that all the above works only considered smooth  {(or even constant)} coefficients.

When the material through which the waves propagate is
heterogeneous,  {the functions $\mu_r$ and $\epsilon_r$ are non constant, and have jumps between materials of different type}. In such a situation, the modest regularity pickup of the exact solution in the scale of Sobolev spaces  is typically lower than $1/2$ and can be arbitrarily close to $0$; see \cite[Theorem 5.1 ]{Bonito_Maxwell_regularity_JMAA_2013}. To the best of our knowledge, all the  {previously mentioned} HDG methods  {have not been proved to converge} for nonsmooth coefficients  {because the standard analysis uses the solution of a dual problem that must have a sufficient regular solution to allow for approximation using appropriate interpolation operators. When the coefficients are piecewise smooth, the appropriate regularity estimates are not available.  Thus even if the true solution is smooth, the error analysis cannot currently be carried out using the methods considered in the previously mentioned papers. The main novelty of this paper is to prove convergence of our HDG method under realistic assumptions on the coefficients and low regularity for the solution of the
	adjoint problem.
	In particular}, we follow the main idea in \cite{Ern_CMAM_2018} to construct a stable and commuting  quasi-interpolation operator, and then  {adapt} the techniques in \cite{Buffa_SINUM_2006}  {(developed to analyze standard DG methods)} to approximate equation \eqref{Maxwell_PDE_ori} using HDG under weak assumptions on the coefficients, see Assumption \ref{Assumptions_coes}.

Furthermore, if we take $q\in H_0^1(\Omega)$ and integrate \eqref{Maxwell_PDE_ori1} multiplied by $\nabla q$ we obtain  $( {\overline{\epsilon_r}}\nabla p,\nabla q)_{\bm L^2(\Omega)}=(\bm f,\nabla q)_{\bm L^2(\Omega)}$ giving an independent system for $p$. Unfortunately, such decomposition does not hold for the discrete system found in all the  above mnetioned HDG methods. 
{A second novel contribution of this paper is to mix the use of HDG spaces for $\bm u$ and a continuous Galerkin (CG) spaces for $p$.  Our method,  {which we term HDG-CG}, retains the flexibility of HDG for the desired field $\bm u$, but uses a smaller space for $p$ which is usually easy to approximate (and for which designing $hp$-spaces is much easier than for edge elements).} Then the system of the Lagrange multiplier is thus SPD and can be solved very efficiently by MG or AMG. In other words, it is good for designing a block preconditioner and therefore, it is good for solving the whole system. 

The outline of the paper is as follows. In \Cref{wellposednessandregularity}, we  {recall} the well posedness and regularity of the Maxwell equations \eqref{Maxwell_PDE_ori}. In \Cref{HDG_mathods}, we set some notation and give the HDG formulation of \eqref{Maxwell_PDE_ori}. The error analysis is given in \Cref{Error_estimates}, we obtain optimal convergence rate for the electric field $\bm u$ and  $\bm \nabla \times \bm u$. Numerical experiments are provided to confirm our theoretical results in \Cref{Numerical_experiments}.  	

\section{Well-posedness and regularity of Maxwell's equations}
\label{wellposednessandregularity}

In this section, we first set some notation which will be used through this paper. Second, we discuss the well-posedness and regularity of Maxwell's equations \eqref{Maxwell_PDE_ori}. Next, we give the  {continuous} inf-sup condition of the mixed form of  \eqref{Maxwell_PDE_ori}  {since we will use it later in our analysis of the} HDG-CG method for  {approximating}  \eqref{Maxwell_PDE_ori}.

For any  bounded domain $\Lambda\subset\mathbb{R}^3$, let $\bm H^{m}(\Lambda)$ denote the usual  $m^{th}$-order Sobolev space of vector functions on $\Lambda$, and $\|\cdot\|_{m, \Lambda}$, $|\cdot|_{m,\Lambda}$  denote the corresponding norm and semi-norm. We use $(\cdot,\cdot){_\Lambda}$ to denote the complex inner product on $\bm L^2(\Lambda)$. Similarly, for the boundary $\partial \Lambda$ of $\Lambda$, we use $\langle\cdot,\cdot\rangle_{\partial\Lambda}$ to denote the $\bm L^2$ inner product on $\partial \Lambda$. We define
\begin{align*}
\bm{H}(\text{curl};\Lambda)&:=\{\bm u\in \bm L^2(\Lambda): \bm \nabla \times \bm u \in \bm L^2({\Lambda}) \}, \\
\bm H_0({\rm curl;\Lambda})&:=\{\bm u\in \bm  H({\rm curl;\Lambda}): \bm n \times \bm u=\bm 0 \textup{ on } \partial \Lambda\},\\
\bm H({\rm div}; \Lambda)&:=\{\bm u\in \bm L^2(\Lambda): \bm \nabla \cdot \bm u \in L^2({\Lambda}) \}, \\
\bm H({\rm div}^0_{\epsilon_r};\Lambda)&:=\{\bm u\in \bm L^2(\Lambda): \bm \nabla\cdot(\epsilon_r\bm u)=0\},
%
%
\end{align*}
where $ \bm n$ is  the  unit outward normal vector on $\partial \Lambda$.

In this paper, we  {allow} the coefficients $\mu_r$ and $\epsilon_r$ of the Maxwell's equations  \eqref{Maxwell_PDE_ori} to be non-smooth. More precisely,  {we assume the following:} 
\begin{assumption}\label{Assumptions_coes}
	The domain $\Omega$ can be decomposed into $N$ subdomains denoted $\Omega_j$, $j=1,2\ldots, N$ such that
	$\bar{\Omega} = \bigcup_{j=1}^N \bar \Omega_j$, $\Omega_i\cap\Omega_j=\emptyset$ if $i\neq j$, and each subdomain $\Omega_j$, $j=1,2\ldots, N$, is connected and has a Lipschitz boundary. Moreover,  
	\begin{description}
		\item[(A)] $\mu_r, \epsilon_r \in W_{\Sigma}^{1, \infty}(\Omega):=\{ v\in  L^\infty(\Omega);  \nabla( v|_{\Omega_i})\in \bm L^\infty(\Omega_i), i=1,2\ldots, N\}$.
		
		\item[(B)] There  {constants} exist $\bar\mu_r, \bar\epsilon_r>0$ such that ${\rm Re}~ (\mu_r)> \bar \mu_r $  and ${\rm Re}~( \epsilon_r)> \bar \epsilon_r$  {a.e. in $\Omega$.}
	\end{description}	
\end{assumption}
We don't require any positivity on the imaginary part of the coefficients in the Assumption (B) although generally $\Im(\epsilon_r)\geq 0$. 

\subsection{Regularity results for Maxwell's equations}
To investigate the regularity of the Maxwell's equations \eqref{Maxwell_PDE_ori}, one needs to characterize the space $\bm H_0({\rm curl};\Omega)$,  $\bm H({\rm div}_{\epsilon_r}^0;\Omega)$ and $\bm L^2(\Omega)$. The next three lemmas give  {several properties} of these spaces, and the proofs of these lemmas can be found in \cite{Hiptmair_electromagnetism_Acta_2002}.
\begin{lemma}[{\cite[Lemma 4.2 ]{Hiptmair_electromagnetism_Acta_2002}}] \label{hip1}
	There exists $s_0>0$ such  $\bm H_0({\rm curl};\Omega)\cap \bm H({\rm div}_{\epsilon_r}^0;\Omega)$
	is continuously embedded in $\bm H^s(\Omega)$ for all $s<s_0$, i.e., the following estimate holds
	\begin{align*}
	\|\bm u\|_{\bm H^s(\Omega)}\le C\left( \|\bm u\|_{\bm L^2(\Omega)}+ \|\bm \nabla\times\bm u\|_{\bm L^2(\Omega)}\right).
	\end{align*}
\end{lemma}

\begin{lemma}[{\cite[Corollary 4.3]{Hiptmair_electromagnetism_Acta_2002}}]\label{hip3}
	The embedding $\bm H_0({\rm curl};\Omega)\cap \bm H({{\rm div}^0_{\epsilon_r}};\Omega)\hookrightarrow  \bm L^2(\Omega)$ is compact.
\end{lemma}

\begin{lemma}[{\cite[Poincar\'{e}-Friedrichs-type inequality, Corollary 4.4]{Hiptmair_electromagnetism_Acta_2002}}] \label{hip2}
	There is a constant $C>0$ depending on $\Omega$ only, such that	for any \[ {\bm u\in}\bm H_0({\rm curl};\Omega)\cap \bm H({\rm div}_{\epsilon_r}^0;\Omega),\] 
	we have
	\begin{align*}
	\|\bm u\|_{\bm L^2(\Omega)}\le C\|\bm \nabla\times\bm u\|_{\bm L^2(\Omega)}.
	\end{align*}
\end{lemma}

In \cite{Bonito_Maxwell_regularity_JMAA_2013}, Bonito et al.  proved the following regularity result with real coefficents $\mu_r$ and $\epsilon_r$ which satisfy Assumption \ref{Assumptions_coes}. It is trival to extend the result to complex coefficents and hence the proof is omitted.
\begin{lemma}[{\cite[Theorem 5.1 ]{Bonito_Maxwell_regularity_JMAA_2013}}] \label{Th1}
	Let $(\bm u, p)\in \bm H_0({\rm curl}, \Omega)\times H_0^1(\Omega)$ be the solution of 
	\begin{subequations}\label{simple_mode_k0}
		\begin{align}
		\bm\nabla  \times (\mu_r^{-1}  \bm\nabla\times\bm u)+\overline{\epsilon_r}\nabla p&=\bm f &\textup{ in }\Omega,\\
		\bm \nabla\cdot(\epsilon_r\bm u)&=0 &\textup{ in }\Omega,\\
		\bm n\times\bm u&=\bm 0&\textup{ on }\partial\Omega,\\
		p&=0&\textup{ on }\partial\Omega.
		\end{align}
	\end{subequations}
	If $\mu_r$ and $\epsilon_r$ satisfy  Assumption \ref{Assumptions_coes}, then there exsits $s>0$ such that 
	\begin{align*}
	\|\bm u\|_{\bm H^s(\Omega)}+	\|\bm \nabla \times \bm u\|_{\bm H^s(\Omega)} \le C\|\bm f\|_{\bm L^2(\Omega)},
	\end{align*}
	where the constant $C$ depends on $s$, $\Omega$, $\mu_r$ and $\epsilon_r$.
\end{lemma}

The  {Maxwell system \eqref{Maxwell_PDE_ori}} we concerned with here  is different to the model in equations \eqref{simple_mode_k0}, hence we need to give a rigorous proof for the next results. We first consider the following positive-definite problem.
\begin{lemma}\label{Th2}
	Let $(\bm u, p)$ be the solution of 
	\begin{subequations}\label{simple_mode_kplus}
		\begin{align}
		\bm\nabla  \times (\mu_r^{-1}  \bm\nabla\times\bm u)+\kappa^2\epsilon_r\bm u+\overline{\epsilon_r}\nabla p&=\bm f &\text{in }\Omega,\label{simple_mode_kplus1}\\
		\bm \nabla\cdot(\epsilon_r\bm u)&=0 &\text{in }\Omega,\label{simple_mode_kplus2}\\
		\bm n\times\bm u&=\bm 0&\text{on }\partial\Omega,\label{simple_mode_kplus3}\\
		p&=0&\text{on }\partial\Omega.\label{simple_mode_kplus4}
		\end{align}
	\end{subequations}
	If $\mu_r$ and $\epsilon_r$ satisfy the Assumption \ref{Assumptions_coes}, then \eqref{simple_mode_kplus} has a unique solution $(\bm u,p)\in \bm H_0({\rm curl};\Omega)\times H_0^1(\Omega)$, and there exsits $s>0$ such that 
	\begin{align*}
	\|\bm u\|_{\bm H^s(\Omega)}+	\|\bm \nabla \times \bm u\|_{\bm H^s(\Omega)} \le C\|\bm f\|_{\bm L^2(\Omega)},
	\end{align*}
	where the constant $C$ depends on  {$s$,} $\Omega$, $\mu_r$ and $\epsilon_r$.
\end{lemma}
\begin{proof}
	We write  \eqref{simple_mode_kplus} in the following varational form: find $(\bm u,p)\in \bm H_0({\rm curl};\Omega)\times H_0^1(\Omega)$ such that
	\begin{align*}
	(\mu_r^{-1}\bm\nabla\times\bm u,\bm\nabla\times\bm v)_{\bm L^2(\Omega)}+
	\kappa^2(\epsilon_r\bm u,\bm v)_{\bm L^2(\Omega)}
	+
	(\nabla p,\epsilon_r\bm v)_{\bm L^2(\Omega)}&=(\bm f,\bm v)_{\bm L^2(\Omega)},\\
	(\epsilon_r\bm u,\nabla q)_{\bm L^2(\Omega)}&=0
	\end{align*} 
	for all $(\bm v,q)\in \bm H_0({\rm curl};\Omega)\times H_0^1(\Omega)$.  The following inf-sup condition holds  simply by taking $\bm w=\nabla p$  {above}:
	\begin{align}\label{inf}
	\sup_{\bm 0\neq \bm w\in\bm H_0({\rm curl};\Omega)}\frac{{\rm Re}~[(\overline{\epsilon_r}\nabla p,\bm w)_{\bm L^2(\Omega)}]}{\|\bm w\|_{\bm L^2(\Omega)}}\ge C\|\nabla p\|_{\bm L^2(\Omega)}.
	\end{align}
	By   \eqref{inf}, the uniqueness  {of $p$} follows immediately. In addition the following coercivity estimate holds  {using \eqref{hip2} and the fact that $\nabla\cdot(\epsilon_r\bm u)=0$:}
	\begin{align}\label{coer}
	\begin{split}
	\hspace{1em}&\hspace{-1em}{\rm Re}~[(\mu_r^{-1}\bm\nabla\times\bm u,\bm\nabla\times\bm u)_{\bm L^2(\Omega)}+
	\kappa^2(\epsilon_r\bm u,\bm u)_{\bm L^2(\Omega)}]\\
	&\ge {\rm Re}~[(\mu_r^{-1}\bm\nabla\times\bm u,\bm\nabla\times\bm u)_{\bm L^2(\Omega)}]\\
	&\ge C\|\bm u\|_{\bm L^2(\Omega)}^2
	\end{split}
	\end{align}
	for all $\bm u\in \bm H_0({\rm curl};\Omega)\cap \bm H({{\rm div}^0_{\epsilon_r}};\Omega)$. Then the existence of a solution follows by the standard LBB theory~\cite{bre94}.
	
	We test \eqref{simple_mode_kplus1} with $\bm u$ and $\eqref{simple_mode_kplus2}$ with $p$, add them together, and take the real part to get
	\begin{align*}
	\left\|\sqrt{{\rm Re}~(\mu_r^{-1})}~ \bm \nabla\times\bm u\right\|^2_{\bm L^2(\Omega)}
	+\kappa^2\|\sqrt{{\rm Re}~(\epsilon_r)}~\bm u\|^2_{\bm L^2(\Omega)}={\rm Re}~ (\bm f,\bm u)_{\bm L^2(\Omega)}.
	\end{align*}
	By  the Cauchy-Schwarz's inequality and the \Cref{hip2},  we have  
	\begin{align}\label{sta-u}
	\|\bm u\|_{\bm L^2(\Omega)}\le C\|\bm f\|_{\bm L^2(\Omega)}.
	\end{align}
	{To obtain a regularity result, we then} rewrite \eqref{simple_mode_kplus} as the following	
	\begin{align*}
	\bm\nabla  \times (\mu_r^{-1}  \bm\nabla\times\bm u) + \overline{\epsilon_r}\nabla p&=\bm f-\kappa^2\epsilon_r\bm u&\textup{in }\Omega,\\
	\bm \nabla\cdot(\epsilon_r\bm u)&=0 &\textup{in }\Omega,\\
	\bm n\times\bm u&=\bm 0&\textup{on }\partial\Omega,\\
	p&=0&\textup{on }\partial\Omega.
	\end{align*}
	By \Cref{Th1},  there exists $s>0$ such that
	\begin{align*}
	\|\bm u\|_{\bm H^s(\Omega)}+	\|\bm \nabla \times \bm u\|_{\bm H^s(\Omega)} \le C\|\bm f-2\kappa^2\epsilon_r\bm u\|_{\bm L^2(\Omega)}.
	\end{align*}
	This, combined with   \eqref{sta-u}, finishes our proof.
\end{proof}
{Next we verify that, under a non-resonance condition, the solution of \eqref{simple_mode_k0} has similar regularity estimates to the positive definite case discssed above:}
\begin{theorem}\label{final_regu1}
	{Suppose that} $\mu_r$ and $\epsilon_r$ satisfy  Assumption \ref{Assumptions_coes} and that  $\kappa^2>0$ is not an eigenvalue of the problem of finding $ {\bm w}\in \bm H_0({\rm curl}, \Omega)$, $\bm w\not=0$, such that
	\begin{align}\label{maxwell_eig}
	\bm \nabla\times(\mu_r^{-1}\bm \nabla\times {\bm w})=\kappa^2\epsilon_r {\bm w}.
	\end{align}
	Then  \eqref{Maxwell_PDE_ori} has a unique solution $\bm u\in \bm H_0({\rm curl};\Omega)$ and there exists $s>0$ such that 
	\begin{align*}
	\|\bm u\|_{\bm H^s(\Omega)}+	\|\bm \nabla \times \bm u\|_{\bm H^s(\Omega)} \le C\|\bm f\|_{\bm L^2(\Omega)},
	\end{align*}
	where the constant $C$ depends on  {$s$,} $\Omega$, $\mu_r$, $\epsilon_r$ and $\kappa$.
\end{theorem}
\begin{proof}
	First, by \Cref{Th2} we know there exists a unique $(\widetilde{\bm u}, {\tilde p})\in \bm H_0({\rm curl};\Omega)\times H^1_0(\Omega)$ such that
	\begin{align*}
	\bm\nabla  \times (\mu_r^{-1}  \bm\nabla\times\widetilde{\bm u})+\kappa^2\epsilon_r\widetilde{\bm u}+\overline{\epsilon_r}\nabla \tilde{p}&=\bm f &\textup{in }\Omega,\\
	\bm \nabla\cdot(\epsilon_r\widetilde{\bm u})&=0 &\textup{in }\Omega,\\
	\bm n\times\widetilde{\bm u}&=\bm 0&\textup{on }\partial\Omega,\\
	\tilde{p}&=0&\textup{on }\partial\Omega,
	\end{align*}
	and
	\begin{align}\label{sta1}
	\|\widetilde{\bm u}\|_{\bm H^s(\Omega)}+	\|\bm \nabla \times \widetilde{\bm u}\|_{\bm H^s(\Omega)} \le C\|\bm f\|_{\bm L^2(\Omega)}.
	\end{align}
	{Furthermore, recalling that we assumed $\rho=0$ in \eqref{Maxwell_PDE_ori2}, we have $\tilde{p}=p$, where $p$ is Lagrange multiplier in \eqref{Maxwell_PDE_ori1} since $\nabla\cdot\overline{\epsilon_r}\nabla p=\nabla\cdot\bm f$ in $\Omega$ and $p=0$ on $\partial \Omega$, with the same equation satisfied by $\tilde{p}$.}
	
	Next, by  {\Cref{Th2}}, for $\bm w\in \bm H({\rm div}_{\epsilon_r}^0;\Omega)$, there exists a unique  {solution denoted} $\bm{K}_{\kappa}\bm w\in \bm H_0({\rm curl};\Omega)$ and $\chi\in H_0^1(\Omega)$ such that
	\begin{subequations}\label{sta20}
		\begin{align}
		\bm\nabla  \times (\mu_r^{-1}  \bm\nabla\times\bm K_{\kappa}\bm w)+\kappa^2\epsilon_r\bm K_{\kappa}\bm w+\overline{\epsilon_r}\nabla\chi&=-2\kappa^2\epsilon_r\bm w&\textup{in }\Omega,\\
		\bm \nabla\cdot(\epsilon_r\bm K_{\kappa}\bm w)&=0,&\textup{in }\Omega,\\
		\bm n\times\bm K_{\kappa}{\bm w}&=\bm 0&\textup{on }\partial\Omega,\\
		\chi&=0&\textup{on }\partial\Omega.
		\end{align}
	\end{subequations}
	Furthermore, since $\nabla\cdot\epsilon_r\bm K_{\kappa}\bm w=0$ we have $\chi=0$ and 
	\begin{align}\label{sta2}
	\|\bm K_{\kappa}\bm w\|_{\bm H^s(\Omega)}+	\|\bm \nabla \times \bm K_{\kappa}\bm w\|_{\bm H^s(\Omega)} \le C_{\kappa}\|\bm w\|_{\bm L^2(\Omega)}.
	\end{align}
	{It is obvious that, if it exists, $\bm u$ must satisfy}
	\begin{align*}
	(\bm I+\bm{K}_{\kappa})\bm u=\widetilde{\bm u}.
	\end{align*}
	{We now prove that the above equation has a unique solution.}  {Let $\{\bm{w}_n\}$ in $\bm L^2(\Omega)$ be a bounded set,   then} \eqref{sta20} and \eqref{sta2} imply that  $\{\bm{K}_{\kappa}\bm{w}_n\}$ is bounded in $\bm H_0({\rm curl};\Omega)\cap\bm H({\rm div}_{\epsilon_r}^0;\Omega)$. Moverover, this set is compact  in $\bm L^2(\Omega)$ by \Cref{hip3}. This proves $\bm{K}_{\kappa}$ is a compact operator from $\bm L^2(\Omega)$ to $\bm L^2(\Omega)$.
	
	Next, we prove that ${\rm ker} \{\bm I+\bm K_{\kappa}\} =0$.  {Let $\bm w\in \bm L^2(\Omega)$ satisfy}
	\begin{align*}
	(\bm I+\bm{K}_{\kappa})\bm w=\bm 0.
	\end{align*}
	{Since $\nabla\cdot\overline{\epsilon_r}\bm{K}_{\kappa}\bm w=0$, this implies $\bm w\in \bm H({\rm div}_{\epsilon_r}^0;\Omega)$ and that it satsifies}
	\begin{align*}
	\bm\nabla  \times (\mu_r^{-1}  \bm\nabla\times\bm w)+\kappa^2\epsilon_r\bm w=-	\bm\nabla  \times (\mu_r^{-1}  \bm\nabla\times\bm K_{\kappa}\bm w)-\kappa^2\epsilon_r\bm K_{\kappa}\bm w=2\kappa^2\epsilon_r\bm w.
	\end{align*}
	Then 
	\begin{align*}
	\bm\nabla  \times (\mu_r^{-1}  \bm\nabla\times\bm w)=\kappa^2\epsilon_r\bm w.
	\end{align*}
	Since $\kappa^2$ is not the eigenvalue of  \eqref{maxwell_eig}, $\bm w=\bm 0$ in $\Omega$ and so
	${\rm ker} \{\bm I+\bm K_{\kappa}\} =0$.
	
	{Having verified that $\bm{K}_{\kappa}$ is compact and $\bm I +\bm{K}_\kappa$ is injective on $\bm L^2(\Omega)$, }
	by the Fredholm alternative theorem~\cite{kress89} we know $(\bm I+\bm{K}_{\kappa})$ is invertible. Hence
	\begin{align}\label{final_regu1_proof1}
	\bm u=(\bm I+\bm K_{\kappa})^{-1}\widetilde{\bm u},
	\end{align}
	and
	\begin{align}\label{final_regu1_proof2}
	\|\bm u\|_{\bm L^2(\Omega)}\le 
	C_{\kappa}\|\widetilde{\bm u}\|_{\bm L^2(\Omega)}.
	\end{align}
	We use  \eqref{sta1}, \eqref{sta2} with $\bm w=\bm u$ to get
	\begin{align*}
	\hspace{1em}&\hspace{-1em}\| \bm u\|_{\bm H^s(\Omega)}+\|\bm \nabla \times \bm u\|_{\bm H^s(\Omega)} \\
	& =  \|\widetilde{\bm u}-\bm K_{\kappa}\bm u\|_{\bm H^s(\Omega)}+	\|\bm \nabla \times (\widetilde{\bm u}-\bm K_{\kappa}\bm u)\|_{\bm H^s(\Omega)} &\textup{ by } \eqref{final_regu1_proof1}\\
	%
	%
	%
	%
	%
	%
	&\le C_{\kappa}	\left(\|\bm f\|_{\bm L^2(\Omega)}+\|\bm u\|_{\bm L^2(\Omega)}\right) &\textup{ by } \eqref{sta1}\textup{ and } \eqref{sta2}\\
	&\le C_{\kappa}	\left(\|\bm f\|_{\bm L^2(\Omega)}+\|\widetilde{\bm u}\|_{\bm L^2(\Omega)}\right) &\textup{ by } \eqref{final_regu1_proof2}\\
	&\le C_{\kappa}	\|\bm f\|_{\bm L^2(\Omega)}&\textup{ by } \eqref{sta1}.
	\end{align*}
	This finishes our proof.
\end{proof}

We test \eqref{Maxwell_PDE_ori1} with $\nabla q$ $(q\in H_0^1(\Omega))$ and use the usual
regularity result for a  second order elliptic problem to get the following result.

\begin{corollary} 
	If $\mu_r$ and $\epsilon_r$ satisfy the Assumption  \ref{Assumptions_coes}, $\bm f\in \bm H({\rm div};\Omega)$. Then there is $s>0$ such that   \eqref{Maxwell_PDE_ori} has a unique solution $p\in H_0^1(\Omega)$. Moreover, the 
	following regularity result holds,
	\begin{align}
	\|p\|_{ H^{1+s}(\Omega)}\le C\|\bm\nabla\cdot\bm f\|_{ L^2(\Omega)},
	\end{align}
	where the constant $C$ depends on $s$, $\Omega$ and $\epsilon_r$. 
\end{corollary}

\subsection{Mixed formulation of Maxwell's equations}

To give our HDG formulation for Maxwell's equations, we need to rewrite them into a mixed form. This is rather standard, and  {we introduce a variable $\bm q$ which is just the scaled magnetic field in electromagnetism.} Let $ \bm q= \mu_r^{-1}\bm \nabla\times\bm u$ in  \eqref{Maxwell_PDE_ori} to get the following mixed form
\begin{subequations}\label{Maxwell_Mixed}
	\begin{align}
	\mu_r \bm q-\bm\nabla\times\bm u&=\bm 0&\textup{ in }\Omega,\label{Maxwell_Mixed1}\\
	\bm\nabla\times\bm q-\kappa^2\epsilon_r\bm u+\overline{\epsilon_r}\nabla p&=\bm f &\textup{ in }\Omega,\label{Maxwell_Mixed2}\\
	\bm \nabla\cdot(\epsilon_r\bm u)&=0&\textup{ in }\Omega,\label{Maxwell_Mixed3}\\
	\bm n\times\bm u&=\bm 0&\textup{ on }\partial\Omega,\label{Maxwell_Mixed4}\\
	p&=0&\textup{ on }\partial\Omega.\label{Maxwell_Mixed5}
	\end{align}
\end{subequations}
Then the mixed weak form of  \eqref{Maxwell_Mixed} is given as follows: find  $(\bm q, \bm u, p)\in \bm L^2(\Omega)\times  \bm H_0({\rm curl};\Omega)\times H^1_0(\Omega)$ such that
\begin{subequations}\label{Maxwell_Mixed_weak}
	\begin{align}
	(\mu_r \bm q, \bm r)_{\bm L^2(\Omega)}-(\bm\nabla\times \bm u, \bm r)_{\bm L^2(\Omega)}&=0, \label{Maxwell_Mixed_weak1}\\
	(\bm q,\bm\nabla\times\bm v)_{\bm L^2(\Omega)}-(\kappa^2\epsilon_r\bm u,\bm v)_{\bm L^2(\Omega)}+(\ \nabla p,{\epsilon_r}\bm v)_{\bm L^2(\Omega)}&=(\bm f,\bm v)_{\bm L^2(\Omega)}, \label{Maxwell_Mixed_weak2}\\
	(\epsilon_r\bm u,\nabla \chi)_{\bm L^2(\Omega)}&=0 \label{Maxwell_Mixed_weak3}
	\end{align}
	for all $(\bm r, \bm v, \chi)\in \bm L^2(\Omega)\times  \bm H_0({\rm curl};\Omega)\times H^1_0(\Omega)$.
\end{subequations}

To shorten lengthy equations,  for all $(\bm q, \bm u, p), (\bm r,\bm v,\chi)\in \bm L^2(\Omega)\times  \bm H_0({\rm curl};\Omega)\times H^1_0(\Omega)$,  we define  {the sesquilinear form}  $\mathscr B^-$ by
\begin{align}\label{def_B}
\begin{split}
\hspace{1em}&\hspace{-1em}\mathscr B^- (\bm q,\bm u, p; \bm r,\bm v,\chi)\\
&=(\mu_r \bm q, \bm r)_{\bm L^2(\Omega)}-(\bm \nabla \times\bm u, \bm r)_{\bm L^2(\Omega)}+(\bm q,\bm\nabla\times\bm v)_{\bm L^2(\Omega)}\\
&\quad 
+(\bar\epsilon_r\nabla p,\bm v)_{\bm L^2(\Omega)} -(\epsilon_r\bm u, \nabla \chi)_{\bm L^2(\Omega)}- (\kappa^2\epsilon_r\bm u,\bm v)_{\bm L^2(\Omega)}.
\end{split}
\end{align}

By the definition of $\mathscr B^-$ in   \eqref{def_B}, we can write the mixed weak form of  \eqref{Maxwell_Mixed} as follows: find $(\bm q,\bm u,p)\in \bm L^2(\Omega)\times  \bm H_0({\rm curl};\Omega)\times H^1_0(\Omega)$ such that 
\begin{align}\label{con_pro}
\mathscr B^- (\bm q,\bm u, p; \bm r,\bm v,\chi) = (\bm f, \bm v)_{\bm L^2(\Omega)}
\end{align}
for all $(\bm r,\bm v,\chi)\in \bm L^2(\Omega)\times \bm H_0({\rm curl};\Omega) \times H^1_0(\Omega)$.

Next, for all $(\bm q, \bm u,p)\in \bm H({\rm curl};\Omega)\times \bm H_0({\rm curl};\Omega)
\times H^1_0(\Omega)$,  we define the semi-norm $\vertiii{\cdot}$ by
\begin{align}\label{trip_morm_defc}
\begin{split}
\vertiii{(\bm q, \bm u, p)} {^2} &= \|\sqrt{{\rm{Re}}~(\mu_r)}~\bm q\|_{\bm L^2(\Omega)}^2 +\|\sqrt{{\rm{Re}}~(\epsilon_r)}~ \bm u \|_{\bm L^2(\Omega)}^2\\
&\quad + \|\bm \nabla\times\bm u\|_{\bm L^2(\Omega)}^2+\|\sqrt{{\rm{Re}}~(\epsilon_r)}~\nabla p\|_{\bm L^2(\Omega)}^2,
\end{split}
\end{align}
where $\mu_r$,  and $\epsilon_r$  {satisfy}  Assumption \ref{Assumptions_coes} and $\kappa>0$. It is easy  to prove that the semi-norm $\vertiii{\cdot}$ is actually a norm.
\begin{lemma}
	$\vertiii{\cdot}$ defines a norm on the space $ \bm L^2(\Omega)\times \bm H_0({\rm curl};\Omega)\times H^1_0(\Omega) $.
\end{lemma}


For all $(\bm q, \bm u, p), (\bm r,\bm v,\chi)\in  \bm L^2(\Omega)\times  \bm H_0({\rm curl};\Omega)\times H^1_0(\Omega)$,  we define the  {the sesquilinear form}  $\mathscr B^+$ by
\begin{align}\label{def_Bplus}
\mathscr B^+ (\bm q,\bm u, p; \bm r,\bm v,\chi) = \mathscr B^- (\bm q,\bm u, p; \bm r,\bm v,\chi) + (\kappa^2+1)(\epsilon_r\bm u,\bm v)_{\bm L^2(\Omega)}.
\end{align}

\begin{lemma}[{Continuous inf-sup condition}] \label{LBBc}
	Let  $\bm \sigma = (\bm q, \bm u, p), \bm \tau = (\bm r, \bm v, \chi)\in \bm L^2(\Omega)\times  \bm H_0({\rm curl};\Omega)\times H^1_0(\Omega)$, we have the following inf-sup conditions
	\begin{align}
	\sup_{\bm 0\neq\bm\tau }\frac{{\rm{Re}}~ [\mathscr B^+(\bm\sigma; \bm\tau)]}{\vertiii{\bm \tau}}\ge C \vertiii{\bm \sigma}.\label{inf_sup_1c}
	\end{align}
\end{lemma}

\begin{proof} 
	First, by the definition of $\mathscr B^+$ in  \eqref{def_B} and  the definition of $\vertiii{\cdot}$ in   \eqref{trip_morm_defc}, and taking  $\bm \tau_1 = (\bm q,\bm u, -p)$ to get
	\begin{align}\label{LBB_proof_1c}
	\begin{split}
	{\rm{Re}} ~[\mathscr B^+ (\bm \sigma;\bm \tau_1)] &=\|\sqrt{{\rm{Re}}~(\mu_r)}~ \bm q\|_{\bm L^2(\Omega)}^2+\|\bm u\|_{\bm L^2(\Omega)}^2,\\
	\vertiii{\bm\tau_1}  &= \vertiii{\bm \sigma}.
	\end{split}
	\end{align}
	
	Next, we take $\bm \tau_2 =(-\bm\nabla\times\bm u_h,\nabla p,0)$ to get
	\begin{align}\label{LBB_proof_2c}
	\begin{split}
	\hspace{1em}&\hspace{-1em} {\rm{Re}} ~[  \mathscr B^+(\bm \sigma;\bm \tau_2)]\\
	&={\rm{Re}} ~[ -(\mu_r \bm q,\bm\nabla\times\bm u)_{\bm L^2(\Omega)}+\|\bm\nabla\times\bm u\|_{\bm L^2(\Omega)}^2+ (\overline{\epsilon_r}\nabla p,\nabla p)_{\bm L^2(\Omega)}+ (\bm u,\nabla p)_{\bm L^2(\Omega)}]\\
	&\ge  -|(\mu_r\bm q,\bm\nabla\times\bm u)_{\bm L^2(\Omega)}|+\|\bm\nabla\times\bm u\|_{\bm L^2(\Omega)}^2 + \|\sqrt{{\rm{Re}}~(\epsilon_r)}~\nabla p\|_{\bm L^2(\Omega)}^2- |(\bm u,\nabla p)_{\bm L^2(\Omega)}|\\
	&\ge \frac{1}{2}\left(\|\bm\nabla\times\bm u_h\|_{\bm L^2(\Omega)}^2+ \|\sqrt{{\rm{Re}}~(\epsilon_r)}~\nabla p\|_{\bm L^2(\Omega)}^2\right) -C_1\left(\|\sqrt{	{\rm{Re}}~(\mu_r)}~\bm q\|_{\bm L^2(\Omega)}^2+\|\bm u\|_{\bm L^2(\Omega)}^2\right),\\
	\hspace{1em}&\hspace{-1em}	\vertiii{\bm \tau_2} \le C_2\vertiii{\bm \sigma}.
	\end{split}
	\end{align}
	
	Finally, we take $\bm \tau = (C_1+1)\bm\tau_1+\bm\tau_2$ and use  \eqref{LBB_proof_1c} and \eqref{LBB_proof_2c} to complete the proof of   \eqref{inf_sup_1c}.
\end{proof}

\section{The HDG method}
\label{HDG_mathods}

To describe the HDG method, we first define some notation. Let $\mathcal{T}_h:=\{K\}$ denote a conforming  {and regular} mesh of $\Omega$, where each element $K$ is  {a tetrahedron}. 
For each $K\in\mathcal{T}_h$, we let $h_K$ be the infimum of the diameters of balls containing $K$ and denote the mesh size $h:=\max_{K\in\mathcal{T}_h}h_K$.  Let
$\partial\mathcal T_h$ denote the set of faces $F\subset \partial K$ of the elements $K\in\mathcal{T}_h$ (i.e. faces of distinct elements are counted separately) and let $\mathcal{F}_h$ denote the set of faces in the mesh $\mathcal T_h$. We denote by $h_F$ the diameter of the face $F$.
We abuse notation by using $\bm \nabla \times$, $\bm \nabla\cdot$ and $\bm \nabla$ for broken \textup{curl}, \textup{div} and gradient operators with respect to the mesh partition $\mathcal{T}_h$, respectively.
To simplify the notation, we also define a function ${\bf h}$ on $\mathcal T_h$, $\partial\mathcal T_h$ and $\mathcal F_h$ which depending on circumstances is defined by:
\begin{align*}
{\bf h}|_K=h_K,\quad\forall K\in \mathcal T_h,\qquad
{\bf h}|_{\partial K}=h_K,\quad\forall K\in \mathcal T_h,\qquad
{\bf h}|_{F}=h_F,\quad\forall F\in \mathcal F_h.
\end{align*}

Next, we list some formulas which will be frequently used in this paper.
\begin{itemize}
	%
	%
	%

	\item[(1)] 
	Let $F\in \mathcal F_h$,  {and let $\nabla_F\cdot$ denote the surface divergence on $F$ where the definition of $\bm \nabla_F\cdot$ can be found in \cite[Section 3.4 (page 48)]{Monk_Maxwell_Book_2003} then the following identity holds for all sufficiently smooth vector functions $\bm v$ defined in a neighborhood of $F$}:	
	\begin{align}
	\bm \nabla_F \cdot  (\bm n\times \bm v)  = -\bm n \cdot (\bm \nabla \times \bm v)|_F.\label{differential_on_surface1}
	\end{align}

	\item[(2)] 
	Let $K$ be an element in the mesh $\mathcal T_h$,
	$\bm u, \bm v\in \bm H(\textup{curl};K)$, $\bm w\in \bm H({\rm div}, F)$,  {$p\in H^1(\partial F)$} and $\langle\!\langle\cdot, \cdot \rangle\!\rangle_{\partial F}$ be the standard  conjugate-linear {$H^{1/2}-H^{-1/2}$ duality pairing } on $\partial F$.  {In addition let $\nabla_F$ denote the surface gradient on $F$.} Then we have 
	\begin{subequations}
		\begin{align} 
		(\bm \nabla\times\bm u, \bm v)_K=\langle\bm n\times\bm u,\bm v\rangle_{\partial K}+(\bm u, \bm \nabla\times \bm v)_K, \label{integration_by_parts1}\\
		-\langle \bm \nabla_F\cdot \bm w, p\rangle_F =
		\langle  \bm w,\nabla_F p \rangle_F
		-\langle\!\langle\bm n_E\cdot \bm w, p\rangle\!\rangle_{\partial F}, \label{integration_by_parts2}
		\end{align}
		where $\bm n$ and $\bm n_E$ are the unit normal to each face of $\partial K$ and each edge of $\partial F$.
	\end{subequations}
	
\end{itemize}

Next, to give the HDG fomulation of  \eqref{Maxwell_PDE_ori}. First we define the following finite element spaces. Let $k\ge 1$, $m=k-1$ or $m=k$, 
\begin{align*}
\bm Q_h&:=\{\bm q_h\in \bm{L}^2(\Omega) :\bm q_h|_K\in[\mathcal P_{m}(K)]^3,\forall K\in \mathcal T_h  \},\\
\bm U_h&:=\{\bm u_h\in \bm{L}^2(\Omega) :\bm u_h|_K\in[\mathcal P_{k}(K)]^3,\forall K\in \mathcal T_h  \},\\
\widehat{\bm U}_h&:=\{\widehat{\bm u}_h\in \bm L^2(\mathcal F_h):\widehat{\bm u}_h|_F\in[\mathcal P_k(F)]^3,\widehat{\bm u}_h\cdot\bm n|_F=0,\forall F\in \mathcal F_h,\bm n\times\widehat{\bm u}_h|_{\partial\Omega}=\bm{0} \},\\
P_h&:=\{p_h\in {H}^1_0(\Omega) :p_h|_K\in\mathcal P_{k+1}(K),\forall K\in \mathcal T_h  \}.
\end{align*}

We can now derive the HDG method for  \eqref{Maxwell_Mixed} by multiplying each equation by the appropriate discrete test
function, integrating element by element and using integration by parts element by element in the usual way (c.f.~\cite{Cockburn_Gopalakrishnan_Lazarov_Unify_SINUM_2009}).
Summing the results over all elements, the  HDG methods seeks  $(\bm q_h,\bm u_h, \widehat{\bm u}_h, p_h)\in  \bm Q_h\times\bm U_h\times  \widehat{\bm U}_h\times P_h$, such that
\begin{subequations}\label{Maxwell_equation_HDG_form_ori}
	\begin{align}
	(\mu_r \bm q_h, \bm r_h)_{\mathcal{T}_h}-(\bm u_h, \bm \nabla\times \bm r_h)_{\mathcal{T}_h}-\langle \bm n\times\widehat{\bm u}_h, \bm r_h \rangle_{\partial\mathcal{T}_h}&=0,\label{Maxwell_equation_HDG_form_ori_1}\\
	(\bm q_h,\bm\nabla\times\bm v_h)_{\mathcal{T}_h}
	+\langle\bm n\times\widehat{\bm  q}_h, {\bm v}_h \rangle_{\partial\mathcal{T}_h}-(\kappa^2\epsilon_r\bm u_h,\bm v_h)_{\mathcal{T}_h}+(\overline{\epsilon_r} \nabla p_h,\bm v_h)_{\mathcal T_h}&=(\bm f,\bm v_h)_{\mathcal{T}_h},\label{Maxwell_equation_HDG_form_ori_2}\\
	(\epsilon_r\bm u_h,\nabla \chi_h)_{\mathcal T_h}&=0,\label{Maxwell_equation_HDG_form_ori_3}\\
	\langle\bm n\times\widehat{ \bm q}_h,\widehat{\bm v}_h \rangle_{\mathcal{F}_h/\partial\Omega}&=0\label{Maxwell_equation_HDG_form_ori_4}
	\end{align}
	for all $(\bm r_h,\bm v_h, \widehat{\bm v}_h,\chi_h)\in \bm Q_h\times\bm U_h\times \widehat{\bm U}_h\times P_h$,  and the choice of $\bm n\times\widehat {\bm q}_h$ follows
	the usual HDG pattern,
	\begin{align}
	\bm n\times\widehat{ \bm q}_h=
	\bm n\times \bm q_h+{\bf h}^{-1}\bm n\times(\bm u_h-\widehat{\bm u}_h)\times\bm n.\label{Maxwell_equation_HDG_form_ori_6}
	\end{align}
\end{subequations}

It  is obvious to see that we can decouple the pressure $p_h$ from the system \eqref{Maxwell_equation_HDG_form_ori} if we take $\bm v_h = \nabla \chi_h$. It is worth mentioning that such a decomposition does not hold for the discrete system of  {other} HDG methods in the literature. We list most of them in \Cref{tab1}, where,  {in the table}, $k$ and $k^-$  {is used as a compact way to} denote the spaces $\mathcal P_k(K)$ and $\mathcal P_{k-1}(F)\oplus \nabla \mathcal {\widetilde P}_{k+1}(F)$, respectively. Here $\widetilde{\mathcal{P}}_{k}(F)$ are the  spaces of polynomials homogeneous of degree $k$ on $F$.
\begin{table}[H]
	\caption{ \label{tab1}Comparison of different HDG methods}
	\centering
	\begin{tabular}{c|c|c|c|c|c|c}
		\Xhline{1pt}
		
		Type &$\bm q_h$ &$\bm u_h$ &$\widehat{\bm u}_h$ &$p_h$ &$\widehat p_h$   &stabilization for $p_h$\\
		\Xhline{1pt}
		
		I\cite{Chen_maxwell_HDG_2017} &$k-1$ &$k$ &$k^-$ &$k-1$ &$k$      &$\langle{\bf h}(p_h-\widehat p_h),q_h-\widehat q_h \rangle_{\partial\mathcal T_h}$\\
		
		II\cite{Chen_maxwell_HDG_2017} &$k-1$ &$k$ &$k^-$ &$k-1$ &$k$      &None\\
		
		III\cite{Chen_Maxwell_HDG_CMAME_2018} &$k$ &$k$ &$k$ &$k+1$ &$k+1$      &$\langle{\bf h}^{-1}(p_h-\widehat p_h),q_h-\widehat q_h \rangle_{\partial\mathcal T_h}$\\

		IV\cite{Chen_M2AN_2019,Lu_hp_HDG_Maxwell_Math_Comp_2017} &$k-1$\textup{ or }$k$ &$k$ &$k$ &$k$ &$k$      &$\langle{\bf h}(p_h-\widehat p_h),q_h-\widehat q_h \rangle_{\partial\mathcal T_h}$\\
		
		V\cite{Du_Arxiv_2019} &$k-1$ &$k$ &$k$ &$k-1$ &$k$      &$\langle{\bf h}(p_h-\widehat p_h),q_h-\widehat q_h \rangle_{\partial\mathcal T_h}$\\
		
		VI\cite{Du_Arxiv_2019} &$k-1$ &$k$ &$k^-$ &$k$ &$k$      &$\langle{\bf h}(p_h-\widehat p_h),q_h-\widehat q_h \rangle_{\partial\mathcal T_h}$\\
		\Xhline{1pt}
	\end{tabular}
\end{table}

{Following} the definition of $\mathscr B^\pm$, we define the discrete sesquilinear form $\mathscr B_h^\pm$  on  the space $\bm Q_h\times\bm U_h\times  \widehat{\bm U}_h\times P_h$ by
\begin{align}\label{def_Bh}
\begin{split}
\hspace{1em}&\hspace{-1em}\mathscr B_h^- (\bm q_h,\bm u_h, \widehat{\bm u}_h, p_h; \bm r_h,\bm v_h, \widehat{\bm v}_h,\chi_h)\\
&=(\mu_r \bm q_h, \bm r_h)_{\mathcal{T}_h}-(\bm u_h, \bm \nabla\times \bm r_h)_{\mathcal{T}_h}-\langle \bm n\times\widehat{\bm u}_h, \bm r_h \rangle_{\partial\mathcal{T}_h}\\
&\quad +(\bm\nabla\times\bm q_h,\bm v_h)_{\mathcal T_h}
+\langle {\bm q}_h,\bm n\times\widehat{\bm v}_h \rangle_{\partial\mathcal T_h}
+(\bar\epsilon_r\nabla p_h,\bm v_h)_{\mathcal T_h}\\
&\quad +\langle{\bf h}^{-1}\bm n\times(\bm u_h-\widehat{\bm u}_h),\bm n\times(\bm v_h-\widehat{\bm v}_h) \rangle_{\partial\mathcal T_h}-(\epsilon_r\bm u_h, \nabla \chi_h)_{\mathcal T_h}- (\kappa^2\epsilon_r\bm u_h,\bm v_h)_{\mathcal{T}_h},\\
\hspace{1em}&\hspace{-1em}\mathscr B_h^+ (\bm q_h,\bm u_h, p_h; \bm r_h,\bm v_h,\chi_h) = \mathscr B_h^- (\bm q_h,\bm u_h, p_h; \bm r_h,\bm v_h,\chi_h) + (\kappa^2+1)(\epsilon_r\bm u_h,\bm v_h)_{\mathcal T_h}.
\end{split}
\end{align}	

Then, we can rewrite the HDG formulation \eqref{Maxwell_equation_HDG_form_ori} in a compact form: find $(\bm q_h,\bm u_h, \widehat{\bm u}_h, p_h)\in \bm Q_h\times\bm U_h\times  \widehat{\bm U}_h\times P_h$ such that 
\begin{align}\label{HDG_formulation}
\mathscr B_h^- (\bm q_h,\bm u_h, \widehat{\bm u}_h, p_h; \bm r_h,\bm v_h, \widehat{\bm v}_h,\chi_h) = (\bm f,\bm v_h)_{\mathcal{T}_h}
\end{align}
for all $(\bm r_h,\bm v_h, \widehat{\bm v}_h,\chi_h)\in \bm Q_h\times\bm U_h\times \widehat{\bm U}_h\times P_h$. 

Although $\mathscr B_h^\pm$ is defined on the space $ \bm Q_h\times\bm U_h\times  \widehat{\bm U}_h\times P_h $,  the notation $\mathscr B_h^\pm(\bm r, \bm v, \bm v, \chi; \\ \bm r_h,\bm v_h, \widehat{\bm v}_h,\chi_h)$ is  meaningful if $(\bm r, \bm v, \chi)\in \bm H({\rm curl};\Omega) \times \bm H_0({\rm curl};\Omega)\times H^1_0(\Omega) $.  Moreover, if  $(\bm q, \bm u, p)\in \bm H({\rm curl};\Omega) \times \bm H_0({\rm curl};\Omega)\times H^1_0(\Omega) $ be the solution of  \eqref{Maxwell_Mixed}, then
\begin{align}\label{HDG_formulation1}
\mathscr B_h^- (\bm q,\bm u, {\bm u}, p; \bm r_h,\bm v_h, \widehat{\bm v}_h,\chi_h) = (\bm f,\bm v_h)_{\mathcal{T}_h}.
\end{align}

The proof of the following lemma is very simple and we omit it here.
\begin{lemma}\label{symmertic_of_B}
	For any $(\bm q_h,\bm u_h, \widehat{\bm u}_h,p_h; \bm r_h, \bm v_h, \widehat{\bm v}_h, \chi_h) \in [\bm Q_h\times\bm U_h\times  \widehat{\bm U}_h\times P_h]^2$, we have 
	\begin{align}\label{commute}
	\mathscr B_h^-(\bm q_h, \bm u_h, \widehat{\bm u}_h, p_h; -\bm r_h, \bm v_h, \widehat{\bm v}_h,-\chi_h) = 	\overline{\mathscr B_h^-(\bm r_h, \bm v_h, \widehat{\bm v}_h,\chi_h;-\bm q_h, \bm u_h, \widehat{\bm u}_h,-p_h)}.
	\end{align}
\end{lemma}

\subsection{Preliminary material}
The approximation of Maxwell's equations by DG methods is  studied in \cite{Monk_Maxwell_CMAME_2002,Chen_maxwell_HDG_2017,Chen_Maxwell_HDG_CMAME_2018,Du_Arxiv_2019,Feng_Maxwell_CMAM_2016,Lu_hp_HDG_Maxwell_Math_Comp_2017,Cockburn_M_decomposition_Part1_Math_Comp_2017} where the coefficients are assumed smooth so that the solution is regular enough for  {a duality based error} analysis, i.e., $\bm u\in H^s(\Omega)$ with $s>1/2$. However, by \Cref{final_regu1}, we only have $\bm u \in \bm H^s(\Omega)$ with $1/2>s>0$ if the coefficients satisfy the Assumption \ref{Assumptions_coes} and the  traditional approach will be fail. Recently, Ern et al.~ {\cite{Ern_M2AN_2017}} derived an error estimate for conforming methods under the Assumption \ref{Assumptions_coes} by  constructing stable, commuting  quasi-interpolation projectors. The idea is to compose the canonical finite element interpolation operators with a mollification technique; also see Sch\"{o}berl \cite{Schoberl_MathComp_2008} and Christiansen \cite{Christiansen_NM_2007} for more details.

In this section, we follow  \cite[Chapter 5]{Monk_Maxwell_Book_2003} to construct  standard interpolation operators $\bm{\mathcal I}_h^{\rm div}$ and  $\bm{\mathcal I}_h^{\rm curl}$, which are only defined on $\bm H^s(\Omega)$ with $s>1/2$; then by a molification operator $\bm{\mathcal K}_h$ which was constructed in \cite{Ern_CMAM_2016,Ern_M2AN_2017}, we can smooth functions in  the space $\bm H^s(\Omega)$ with $s>0$ into $\bm H^{s'}(\Omega)$ with $s'>1/2$;  finally define the quasi-interpolations by
\begin{align*}
\bm{\mathcal J}_h^{\rm div}=((\bm{\mathcal I}_h^{\rm div}\bm{\mathcal K}_h|_{\bm V_h^{\rm div}})^{-1}\bm{\mathcal I}_h^{\rm div}\bm{\mathcal K}_h,\qquad
\bm{\mathcal J}_h^{\rm curl}=(\bm{\mathcal I}_h^{\rm curl}\bm{\mathcal K}_h|_{\bm V_h^{\rm curl}}
)^{-1}\bm{\mathcal I}_h^{\rm curl}\bm{\mathcal K}_h.
\end{align*}

First, we define some spaces which will be useful in our analysis. For the convenience, we define 
\begin{align*}
\bm V_h^{\rm curl}&:=\bm U_h\cap \bm H({\rm curl};\Omega), \quad \bm V_h^{\rm div}:=\bm{\mathcal D}_k(\mathcal T_h)\cap \bm H({\rm div};\Omega),\\
\bm{\mathcal D}_k(K)&:=[\mathcal P_{k-1}(K)]^3+\widetilde{\mathcal{P}}_{k-1}(K)\bm x, \quad \bm{\mathcal D}_k(F):=[\mathcal P_{k-1}(F)]^3+\widetilde{\mathcal{P}}_{k-1}(F)\bm x.
\end{align*}
where $\widetilde{\mathcal{P}}_{k}(K)$ and $\widetilde{\mathcal{P}}_{k}(F)$ are the  spaces of polynomials homogeneous of degree $k$ on $K$ and $F$, respectively.

Next, we define $\bm{\mathcal I}_h^{\rm div}$ to be the divergence conforming interpolation of the  first family from $\bm H^s(\Omega)\to \bm V_h^{\rm div}$ by \cite[Page 328. Definition 5]{Nedelec_mixed_NM_1980} and \cite[Page 119. Definition 5.14]{Monk_Maxwell_Book_2003}  {which is defined element by element via}
\begin{subequations}\label{div}
	\begin{align}
	(\bm{\mathcal I}_{h}^{\rm div}  \bm v, \bm q_{k-2})_K&=(\bm v,\bm q_{k-2})_K\qquad\quad\forall \bm q_{k-2}\in\bm {\mathcal P}_{k-2}(K),\\
	\langle\bm n\cdot\bm{\mathcal I}_{h}^{\rm div} \bm v,q_{k-1}\rangle_F&=
	\langle\bm n\cdot \bm v,q_{k-1}\rangle_F\qquad\forall q_{k-1}\in\mathcal P_{k-1}(F)
	\end{align}
	for all faces $F\subset \partial K$, where $s>1/2$ and $\bm v\in \bm H^s({\rm div};\Omega)$.
\end{subequations}
Moreover, we define $\bm{\mathcal I}_h^{\rm curl}$ be the curl conforming interpolation of the  second family from $\bm H^s({\rm curl};\Omega)\to \bm V_h^{\rm curl}$  {element by element} by
\begin{subequations}\label{curl}
	\begin{align}
	(\bm{\mathcal I}_{h}^{\rm curl}\bm v,\bm q_{k-2})_K&=(\bm v,\bm q_{k-2})_K\qquad\qquad\qquad \forall \bm q_{k-2}\in \bm{\mathcal D}_{k-2}(K),\label{curl1}\\
	\langle\bm n\times\bm{\mathcal I}_{h}^{\rm curl}\bm v\times \bm n, \bm q_{k-1}\rangle_F&=
	\langle\bm n\times\bm v\times \bm n,\bm q_{k-1}\rangle_F
	\qquad\forall \bm q_{k-1}\in  \bm{\mathcal D}_{k-1}(F),\label{curl2}\\
	\langle\!\langle\bm{\mathcal I}_{h}^{\rm curl}\bm v\cdot\bm{\tau},q_{k}\rangle\!\rangle_E&=
	\langle\! \langle\bm v\cdot\bm {\tau},q_{k-1}\rangle\!\rangle_E
	\qquad\qquad\forall q_{k}\in\mathcal P_{k}(E)\label{curl3}
	\end{align}	
	for all faces $F\subset \partial K$ and all edges $E\subset \partial F$, and 
	$\bm v\in \bm H^s({\rm curl};\Omega)$ with $s>1/2$.
\end{subequations}

{The following lemma shows that the usual commutativity properties hold for the combined first and second kind interpolants used here:}

\begin{lemma}\label{can_in}
	For $s>1/2$ and $\bm v \in \bm H^{s}(\rm{curl}, K)$, let
	$\bm{\mathcal I}_h^{\rm div}$ and  $\bm{\mathcal I}_h^{\rm curl}$ define in  \eqref{div} and \eqref{curl}, respectively. Then we have the following commutativity property
	\begin{align*}
	\bm \nabla\times (\bm{\mathcal I}_{h}^{\rm curl}\bm v)=\bm{\mathcal I}_{h}^{\rm div}(\bm \nabla\times \bm v).
	\end{align*}
\end{lemma}	
\begin{proof}
	For all $\bm q_{k-2}\in\bm {\mathcal P}_{k-2}(K)$,  we get
	\begin{align*}
	(\bm \nabla\times(\bm{\mathcal I}_{h}^{\rm curl}\bm v),\bm q_{k-2})_K&=
	(\bm{\mathcal I}_{h}^{\rm curl}\bm v,\bm \nabla\times\bm q_{k-2})_K+
	\langle\bm n\times(\bm{\mathcal I}_{h}^{\rm curl}\bm v),\bm q_{k-2}\rangle_{\partial K}&\textup{ by } \eqref{integration_by_parts1}\\
	&=(\bm v,\bm \nabla\times\bm q_{k-2})_K+
	\langle\bm n\times\bm v,\bm q_{k-2}\rangle_{\partial K}&\textup{ by } \eqref{div}\\
	&=(\bm \nabla\times\bm v,\bm q_{k-2})_K&\textup{ by } \eqref{integration_by_parts1}.
	\end{align*}
	Next,  for all $ q_{k-1}\in\mathcal P_{k-1}(F)$, we have 
	\begin{align*}
	\hspace{1em}&\hspace{-1em}\langle\bm n\cdot (\bm \nabla\times(\bm{\mathcal I}_{h}^{\rm curl}\bm v)),q_{k-1}\rangle_F\\
	&=-\langle\nabla_F\cdot(\bm n\times\bm{\mathcal I}_{h}^{\rm curl}\bm v),q_{k-1}\rangle_F &\textup{ by } \eqref{differential_on_surface1}\\
	&=
	\langle\bm n\times\bm{\mathcal I}_{h}^{\rm curl}\bm v,\nabla_F q_{k-1}\rangle_F
	-\langle\!\langle\bm n_E\cdot(\bm n\times\bm{\mathcal I}_{h}^{\rm curl}\bm v),q_{k-1}\rangle\!\rangle_{\partial F} &\textup{ by } \eqref{integration_by_parts2}\\
	&=
	\langle\bm n\times\bm{\mathcal I}_{h}^{\rm curl}\bm v,\nabla_F q_{k-1}\rangle_F
	-\langle\!\langle\bm{\tau}\cdot\bm{\mathcal I}_{h}^{\rm curl}\bm v,q_{k-1}\rangle\!\rangle_{\partial F}
	&
	\\
	&=\langle\bm n\times\bm v,\nabla_F q_{k-1}\rangle_F
	-\langle\!\langle\bm{\tau}\cdot\bm v,q_{k-1}\rangle\!\rangle_{\partial F}&\textup{ by } \eqref{curl2} \textup{ and }\eqref{curl3}\\
	&=\langle\bm n\times\bm v,\nabla_F q_{k-1}\rangle_F
	-\langle\!\langle\bm n_E\cdot(\bm n\times\bm v),q_{k-1}\rangle\!\rangle_{\partial F}&\\
	&=-\langle\nabla_F\cdot(\bm n\times\bm v),q_{k-1}\rangle_F&\textup{ by } \eqref{integration_by_parts2}\\
	&=\langle\bm n\cdot(\bm \nabla\times\bm v),q_{k-1}\rangle_F&\textup{ by } \eqref{differential_on_surface1}.
	\end{align*}
	By the above arguments, we conclude  {the following}
	\begin{subequations}\label{curl22}
		\begin{align}
		(\bm \nabla\times (\bm{\mathcal I}_{h}^{\rm curl}\bm v),\bm q_{k-2})_K&=(\bm \nabla\times\bm v,\bm q_{k-2})_K\qquad\quad~~~\forall \bm q_{k-2}\in \bm {\mathcal P}_{k-2}(K),\\
		\langle\bm n\cdot(\bm \nabla\times (\bm{\mathcal I}_{h}^{\rm curl}\bm v)),q_{k-1}\rangle_F&=
		\langle\bm n\cdot( \bm \nabla\times\bm v),q_{k-1}\rangle_F
		\qquad\forall q_{k-1}\in\mathcal P_{k-1}(F).
		\end{align}
	\end{subequations}
	By  \eqref{div}, one has
	\begin{subequations}\label{div22}
		\begin{align}
		(\bm{\mathcal I}_{h}^{\rm div} (\bm \nabla\times\bm v),\bm q_{k-2})_K&=( \bm \nabla\times\bm v,\bm q_{k-2})_K\qquad\qquad\forall \bm q_{k-2}\in\bm {\mathcal P}_{k-2}(K),\\
		\langle\bm n\cdot\bm{\mathcal I}_{h}^{\rm div}  (\bm \nabla\times\bm v),q_{k-1}\rangle_F&=
		\langle\bm n\cdot (\bm \nabla\times\bm v),q_{k-1}\rangle_F
		\qquad\forall q_{k-1}\in\mathcal P_{k-1}(F),
		\end{align}
	\end{subequations}
	{Using the unisolvence} of   \eqref{div22} and \eqref{curl22} leads to our desired result.
	
\end{proof}

By \Cref{can_in} and \cite[Theorem 6.5]{Ern_CMAM_2016},
\cite[Cororally 5.4]{Ern_M2AN_2017} we now get the following lemma.
\begin{lemma}[{[Stable, commuting projection]}]\label{stable_commute}
	There exist quasi-interpolation operator $\bm{\mathcal J}_{h}^{\rm curl}:\bm L^1(\Omega)\to {\bm V}_{h}^{\rm curl}$ and $\bm{\mathcal J}_{h}^{\rm div}\to {\bm V}^{\rm div}_{h}$ such that
	\begin{description}
		
		\item[(1)] $\bm U_h\cap \bm H_0({\rm curl}, \Omega)$ is pointwise invariant under $\bm{\mathcal J}_{h}^{\rm curl}$.
		
		\item[(2)] $\bm{\mathcal J}_{h}^{\rm curl}\bm v\in  \bm H_0({\rm curl};\Omega)\textup{ if }\bm v\in \bm H_0({\rm curl};\Omega)$.
		
		\item[(3)] For $p\in [1,\infty]$, there holds the stability
		\begin{align*}
		\|\bm{\mathcal J}_{h}^{\rm curl}\|_{\mathcal L(\bm L^p(\Omega);\bm L^p(\Omega))}&\le C,\\
		\|\bm{\mathcal J}_{h}^{\rm div}\|_{\mathcal L(\bm L^p(\Omega);\bm L^p(\Omega))}&\le C.
		\end{align*}
		\item[(4)] For $p\in [1,\infty]$ and  $s\in [0,k+1]$, we have the following estimates
		\begin{subequations}
			\begin{align}
			\|\bm{\mathcal J}_{h}^{\rm div}\bm v-\bm v\|_{\bm L^p(\Omega)}& \le C\inf_{\bm v_h\in {\bf V}_h^{\rm div}}	\|\bm v_h-\bm v\|_{\bm L^p(\Omega)}\le Ch^s|\bm v|_{s,p},\label{app_div}\\
			\|\bm{\mathcal J}_{h}^{\rm curl}\bm v-\bm v\|_{\bm L^p(\Omega)}& \le C \inf_{\bm v_h\in {\bm V}_{h}^{\rm curl}}	\|\bm v_h-\bm v\|_{\bm L^p(\Omega)}\le Ch^s|\bm v|_{s,p}.\label{app_cul}
			\end{align}
		\end{subequations}

		\item[(5)] For $\bm v\in \bm H({\rm curl};\Omega)$, the following commute property holds
		\begin{align}\label{curl_div_commute}
		\bm \nabla\times (\bm{\mathcal J}_{h}^{\rm curl}\bm v)=\bm{\mathcal J}_{h}^{\rm div}	 (\bm \nabla\times \bm v).
		\end{align}
	\end{description}
	
\end{lemma}

\begin{lemma}[Discrete Helmholtz decomposition]\label{Helmholtz_discrete} 
	For all $\bm v_h\in \bm U_h\cap \bm H_0({\rm curl};\Omega)$, there exist unique functions $\bm z_{h}\in \bm U_h\cap\bm H_0({\rm curl};\Omega)$ and $\xi_h\in  P_h$ such that for all $\chi_h\in  P_h$ we have 
	\begin{align}\label{decomp}
	\bm v_h=\bm z_{h}+\nabla \xi_h, \quad (\epsilon_r\bm z_{h},\nabla \chi_h)_{\mathcal T_h}=0.
	\end{align}
	Moreover,  the following stability results hold
	\begin{align}\label{boud}
	\|\nabla\xi_h\|_{\mathcal T_h}\le  C\|\bm v_h\|_{\mathcal T_h},\qquad 
	\|\bm z_h\|_{\mathcal T_h}\le  C\|\bm v_h\|_{\mathcal T_h}.
	\end{align}
	
\end{lemma}	
\begin{proof} 
	For all $\chi_h\in P_h$, let $\xi_h \in P_h $ be the solution of 
	\begin{align}\label{Helmholtz_decomposition_proof1}
	(\epsilon_r\nabla\xi_h,\nabla \chi_h)_{\mathcal T_h}=(\epsilon_r \bm v_h,\nabla \chi_h)_{\mathcal T_h}.
	\end{align}
	The system \eqref{Helmholtz_decomposition_proof1} is well-defined for any $\bm v_h\in \bm U_h$, hence we have 
	\begin{align}\label{Helmholtz_decomposition_proof2}
	\|\nabla\xi_h\|_{\mathcal T_h}\le C \|\bm v_h\|_{\mathcal T_h}.
	\end{align}
	We take $\bm z_h=\bm v_h-\nabla\xi_h$, then  \eqref{decomp}  and $	\|\bm z_h\|_{\mathcal T_h}\le  C\|\bm v_h\|_{\mathcal T_h}$ hold.  Finally,  $\bm z_h$ is  unique since $\xi_h$ is the unique solution of  \eqref{Helmholtz_decomposition_proof1}.
\end{proof}

The proof of the following lemma with smooth coefficients $\mu_r$ and $\epsilon_r$ is given in \cite[Lemma 4.5]{Hiptmair_electromagnetism_Acta_2002} and \cite[Corollary 4.4]{Houston_Maxwell_NM_2005}.  {We extend the result}  to allow piecewise smooth coefficients.

\begin{lemma} \label{H_curl_app}
	For any $\bm v_h\in \bm U_h\cap \bm H_0({\rm curl};\Omega)$ satisfies
	\begin{align*} 
	(\epsilon_r \bm v_h,\nabla \chi_h)_{\bm L^2(\Omega)}=0,
	\end{align*}
	where $\chi_h\in P_h$. Let $\bm{\Theta}\in \bm H_0({\rm curl};\Omega)\cap  \bm X_0$  be the solution of
	\begin{align*}
	\bm \nabla\times\bm{\Theta}=\bm \nabla\times\bm v_h,
	\end{align*}
	then, we have the following stability estimate for some $s>0$:
	\begin{align}\label{theta_bound}
	\|\bm{\Theta}\|_{\bm H^{s}(\Omega)}\le C\|\bm \nabla\times\bm v_h\|_{\bm L^2(\Omega)},
	\end{align}
	and  the approximation property
	\begin{align}\label{theta_app}
	\|\bm v_h-\bm{\Theta}\|_{\bm L^2(\Omega)}\le Ch^{s}\|\bm \nabla\times\bm v_h\|_{\bm L^2(\Omega)}.
	\end{align}	
\end{lemma}

\begin{proof}
	First, we prove the existence and uniqueness of $\bm\Theta$. For any given $\bm v_h$, we consider the following mixed problem: find $(\bm{\Theta},p)\in \bm H_0({\rm curl};\Omega)\times H^1_0(\Omega)$ such that
	\begin{subequations}\label{ax}
		\begin{align}
		(\bm\nabla\times\bm{\Theta},\bm\nabla\times\bm{w})_{\bm L^2(\Omega)}+(\overline{\epsilon_r}\nabla p,\bm w)_{\bm L^2(\Omega)}&=(\bm\nabla\times\bm v_h,\bm\nabla\times\bm w)_{\bm L^2(\Omega)},\label{ax1}\\
		(\epsilon_r\bm\Theta, \nabla q)_{\bm L^2(\Omega)}&=0\label{ax2}
		\end{align}
	\end{subequations}
	for all $(\bm{w},q)\in \bm H_0({\rm curl};\Omega)\times H^1_0(\Omega)$.  \Cref{hip2} guarantees the coercivity of $(\bm\nabla\times\bm{\Theta},\bm\nabla\times\bm{\Theta})_{\bm L^2(\Omega)}$ on $\bm H_0({\rm curl};\Omega)\cap \bm H({\rm div}^0_{\epsilon_r};\Omega)$.  The following inf-sup condition holds true simply by taking $\bm w=\nabla p$:
	\begin{align}\label{inf-sup}
	\sup_{\bm 0\neq \bm w\in\bm H_0({\rm curl};\Omega)}\frac{{\rm Re}~[(\overline{\epsilon_r}\nabla p,\bm w)_{\bm L^2(\Omega)}]}{\|\bm w\|_{\bm L^2(\Omega)}}\ge C\|\nabla p\|_{\bm L^2(\Omega)}.
	\end{align}
	Thus,   \eqref{inf-sup} leads to the  {existence and }uniqueness of  {solutions of} \eqref{ax}. Take $\bm w=\nabla p$ in  \eqref{ax1} to get $p=0$ and obtain
	\begin{align}
	\bm \nabla\times\bm{\Theta}=\bm \nabla\times\bm v_h.
	\end{align}

	Next, we prove  \eqref{theta_app}. By the \Cref{hip1,hip2}, we get the  {boundedness result}  of \eqref{theta_bound} and 
	\begin{align*}
	\bm \nabla\times(\bm v_h-\bm{\mathcal J}_h^{\rm curl}\bm{\Theta})= \bm 
	\nabla\times(\bm{\mathcal J}_h^{\rm curl}(\bm v_h-\bm{\Theta})) = \bm{\mathcal J}_h^{\rm div}(\bm \nabla\times(\bm v_h-\bm{\Theta}))=\bm 0,
	\end{align*}
	where we used the first property in \Cref{stable_commute}. Then, there exists $\xi_h\in P_h$ such that
	\begin{align*}
	\bm v_h-\bm{\mathcal J}_h^{\rm curl}\bm{\Theta}=\nabla \xi_h.
	\end{align*}
	Since $\bm \Theta \in \bm H({\rm div}_{\epsilon_r}^0;\Omega) $, then $(\epsilon_r\bm \Theta,  \nabla \xi_h)_{\mathcal T_h}=0$. This gives
	\begin{align*}
	\|\bm v_h-\bm{\Theta}\|^2_{\bm L^2(\Omega)} &\le C {\rm Re}~ [(\epsilon_r(\bm v_h-\bm{\Theta}),\bm v_h-\bm{\Theta})_{\bm L^2(\Omega)}]\\
	&= C {\rm Re}~ [(\epsilon_r(\bm v_h-\bm{\Theta}),\bm v_h-\bm{\mathcal J}_h^{\rm curl}\bm{\Theta}+\bm{\mathcal J}_h^{\rm curl}\bm{\Theta}-\bm{\Theta})_{\bm L^2(\Omega)}]\\
	&= C {\rm Re}~ [ (\epsilon_r (\bm v_h-\bm{\Theta}),\bm{\mathcal J}_h^{\rm curl}\bm{\Theta}-\bm{\Theta})_{\bm L^2(\Omega)}].
	\end{align*}
	By  \eqref{app_cul} in \Cref{stable_commute} and   \eqref{theta_app}, we have 
	\begin{align*}
	\|\bm v_h-\bm{\Theta}\|_{\bm L^2(\Omega)}\le C\|\bm{\mathcal J}_h^{\rm curl}\bm{\Theta}-\bm{\Theta}\|_{\bm L^2(\Omega)}\le Ch^{s}\|\bm{\Theta}\|_{\bm H^{s}(\Omega)}
	\le  Ch^{s}\|\bm \nabla\times\bm{v}_h\|_{\bm L^2(\Omega)}.
	\end{align*}
	This finishes our proof.	
\end{proof}

{The final result of this section is ubiquitous in the analysis of DG methods for Maxwell's equations:}

\begin{lemma}[c.f {\cite[Proposition 4.5]{Houston_Maxwell_NM_2005}}]\label{vc} 
	For all $\bm u_h\in \bm U_h$, there exists a $\bm u_h^{\rm c}\in \bm U_h\cap \bm H_0({\rm curl};\Omega)$ such that
	\begin{align}\label{est_vc}
	\|\bm u_h-\bm u_h^{\rm c}\|_{\mathcal T_h}+h	\|\bm \nabla\times(\bm u_h-\bm u_h^{\rm c})\|_{\mathcal T_h}\le C\|{\bf h}^{\frac{1}{2}}\bm n\times[\![\bm u_h]\!]\|_{\mathcal F_h}.
	\end{align}
\end{lemma}

\subsection{Stability  {of the coercive discrete problem}}

Next, for all $(\bm q_h,\bm u_h, \widehat{\bm u}_h, p_h)\in \bm Q_h\times \bm U_h\times \widehat{\bm U}_h\times P_h$,  we define
\begin{align}\label{trip_morm_def}
\begin{split}
\vertiii{(\bm q_h, \bm u_h, \widehat{\bm u}_h, p_h)}_{h}&= \|\sqrt{{\rm Re}~(\mu_r)}~\bm q_h\|_{\mathcal T_h}^2 +\|\sqrt{{{\rm Re}~(\epsilon_r)}}~ \bm u_h \|^2_{\mathcal T_h}+ \|\bm \nabla\times\bm u_h\|_{\mathcal T_h}^2\\
&\quad +\|{\bf h}^{-\frac{1}{2}}\bm n\times(\bm u_h-\widehat{\bm u}_h)\|_{\partial\mathcal T_h}^2+\|\sqrt{{\rm{Re}}~(\epsilon_r)}~\nabla p_h\|_{\mathcal T_h}^2,
\end{split}
\end{align}
where $\mu_r$, $\kappa$  and $\epsilon_r$ were defined in the Assumption \ref{Assumptions_coes}. Next, the proof of the  following  lemma is trivial, hence we omit it here.
\begin{lemma}
	$\vertiii{\cdot}_h$ defines a norm on the space $ \bm Q_h\times \bm U_h\times \widehat{\bm U}_h\times P_h$.
\end{lemma}

By the definition of $\vertiii{\cdot}_h$ in  \eqref{trip_morm_def}, for all $(\bm r, \bm v,\chi)\in \bm L^2(\Omega)\times \bm H_0(\rm{curl}, \Omega)\times H_0^1(\Omega)$,  it is easy to see that $\vertiii{(\bm r, \bm v, \bm n\times \bm v\times \bm n, p)}_{h}$ is well defined  {since $\bm n\times (\bm v-\bm n\times\bm v\times\bm n)=0$ on $\partial{\mathcal T}_h$.} This will be used frequently in the error analysis.

Next, by the Cauchy-Schwarz inequality and the triangle inequality we have the following  {result showing the} boundness of 	$\mathscr B_h^\pm$.
\begin{lemma}[Boundness of 	$\mathscr B_h^\pm$]
	Let $\bm \sigma_h = (\bm q_h,\bm u_h, \widehat{\bm u}_h,p_h)$, $\bm \tau_h = (\bm r_h, \bm v_h, \widehat{\bm v}_h, \chi_h) \in \bm Q_h\times\bm U_h\times  \widehat{\bm U}_h\times P_h$,  then we have
	\begin{align}\label{boundnessofBminus0}
	|\mathscr B_h^\pm(\bm \sigma_h;\bm \tau_h)|\le C\vertiii{\bm \sigma_h}_h\vertiii{\bm \tau_h}_h.
	\end{align}
\end{lemma}

It is worth mentioning that the boundedness of 	$\mathscr B_h^\pm$ also holds  if $\bm \sigma_h$ is replced by $\bm\sigma=(\bm q, \bm u, \bm n\times \bm u\times \bm n, p)$ where $(\bm q, \bm u, p)\in \bm H({\rm curl};\Omega)\times \bm H_0({\rm curl};\Omega)\times H^1_0(\Omega)$,  {because} $\bm n\times (\bm u- \bm n\times \bm u\times \bm n)=\bm 0$ on each face $F\in \mathcal F_h$. Then  we have the following inequality
\begin{align}\label{boundnessofBminus1}
|\mathscr B_h^\pm(\bm\sigma;\bm \tau_h)|\le C\left(\| \bm q\|_{\mathcal T_h} +\|\bm u\|_{\mathcal T_h}+\|\bm \nabla\times \bm u\|_{\mathcal T_h}+\|\nabla p\|_{\mathcal T_h}\right)\vertiii{\bm \tau_h}_h.
\end{align}
By the inequalities \eqref{boundnessofBminus0} and \eqref{boundnessofBminus1},  {provided $\bm u_h\in H_0({\rm curl};\Omega)\cap \bm U_h$,} we have 
\begin{align}\label{boundnessofBminus}
\begin{split}
\hspace{1em}&\hspace{-1em}|\mathscr B_h^\pm(\bm\sigma - \bm \sigma_h;\bm \tau_h)|\\
&\le C\left(\| \bm q-\bm q_h\|_{\mathcal T_h} +\|\bm u-\bm u_h\|_{\mathcal T_h}+\|\bm \nabla\times( \bm u-\bm u_h)\|_{\mathcal T_h}+\|\nabla (p-p_h)\|_{\mathcal T_h}\right)\vertiii{\bm \tau_h}_h.
\end{split}
\end{align}

{In preparation for the proof of the following lemma, for any} $\bm u_h\in\bm U_h$ let  $\bm u_h^{\rm c} \in \bm U_h\cap \bm H_0({\rm curl};\Omega)$ by the conforming function in \Cref{vc},  then by \Cref{Helmholtz_discrete}, there exist $\bm y_h\in \bm U_h\cap \bm H_0({\rm curl};\Omega)$ and $\eta_h\in P_h$, such that for all $\chi_h\in P_h$ we have 
\begin{align}\label{decom1}
\bm u_h^{\rm c}=\bm y_h+\nabla \eta_h, \quad (\epsilon_r \bm y_h,\nabla \chi_h)_{\mathcal T_h}=0.
\end{align}
{We can now prove the discrete analogue of \eqref{LBBc}.}

\begin{lemma}[Discrete inf-sup conditions] \label{LBB}
	For all $\bm \sigma_h = (\bm q_h,\bm u_h, \widehat{\bm u}_h, p_h) \in  \bm \Sigma_h = \bm Q_h\times\bm U_h\times  \widehat{\bm U}_h\times P_h $, we have the following inf-sup conditions
	\begin{align}
	\sup_{\bm 0\neq \bm \tau_h\in\bm \Sigma_h}\frac{ {\rm{Re}} ~ [\mathscr B_h^+ (\bm \sigma_h;\bm \tau_h)]}{\vertiii{\bm \tau_h}_{h}}\ge C \vertiii{\bm \sigma_h}_{h}.\label{inf_sup_1}
	\end{align}
\end{lemma}

\begin{proof} 	
	First, we take $\bm \tau_1 = (\bm q_h,\bm u_h, \widehat{\bm u}_h, -p_h)$, then by the definition of $\mathscr B_h^+$ in   \eqref{def_B}, integration by parts and  the definition of $\vertiii{\cdot}_h$ in  \eqref{trip_morm_def} we get
	\begin{align}\label{LBB_proof_1}
	\begin{split}
	{\rm{Re}} ~[ \mathscr B_h^+ (\bm \sigma_h; \bm \tau_1)]&=\|\sqrt{	{\rm{Re}}~(\mu_r)}~ \bm q_h\|_{\mathcal T_h}^2+\|{\bf h}^{-\frac{1}{2}}\bm n\times(\bm u_h-\widehat{\bm u}_h)\|_{\partial\mathcal T_h}^2+\|\kappa\sqrt{{\rm{Re}}~(\epsilon_r)}~\bm u_h\|_{\mathcal{T}_h}^2\\
	\vertiii{\bm \tau_1}_h & = \vertiii{\bm \sigma_h}_h.
	\end{split}	   
	\end{align}
	
	Second, we take $\bm \tau_2 =(-\bm \nabla\times\bm u_h,\bm 0,\bm 0,0)$ to get
	\begin{align}\label{LBB_proof_2}
	\begin{split}
	\hspace{1em}&\hspace{-1em} {\rm{Re}} ~ [  \mathscr B_h^+(\bm \sigma_h;\bm \tau_2)]\\
	&={\rm{Re}} ~[ -(\mu_r\bm q_h,\bm\nabla\times\bm u_h)_{\mathcal T_h}+\|\bm \nabla\times\bm u_h\|_{\mathcal T_h}^2-\langle
	\bm n\times(\bm u_h-\widehat{\bm u}_h),\bm\nabla\times\bm u_h
	\rangle_{\partial\mathcal T_h}]\\
	&\ge  -|(\mu_r\bm q_h,\bm \nabla\times\bm u_h)_{\mathcal T_h}|+\|\bm\nabla\times\bm u_h\|_{\mathcal T_h}^2-|\langle
	\bm n\times(\bm u_h-\widehat{\bm u}_h),\bm\nabla\times\bm u_h
	\rangle_{\partial\mathcal T_h}|\\
	&\ge \frac{1}{2}\|\bm\nabla\times\bm u_h\|_{\mathcal T_h}^2 -C_1\left(\|\sqrt{	{\rm{Re}}~(\mu_r)}~\bm q_h\|_{\mathcal T_h}^2+\|{\bf h}^{-\frac{1}{2}}\bm n\times(\bm u_h-\widehat{\bm u}_h)\|_{\partial\mathcal T_h}^2\right),\\
	\vertiii{\bm \tau_2}_h &\le C_2\vertiii{\bm \sigma_h}_h.
	\end{split}
	\end{align}
	
	Third,  we take $\bm \tau_3 =(\bm 0,\nabla p_h,\bm n\times \nabla p_h\times\bm n,0)$ to get 
	\begin{align}\label{LBB_proof_4}
	\begin{split}
	{\rm{Re}} ~ [  \mathscr B_h^+(\bm \sigma_h;\bm \tau_3)]&=\|\sqrt{{\rm Re}~(\epsilon_r)}~\nabla p_h\|_{\mathcal T_h}^2+ {\rm{Re}} ~ [(\kappa^2\epsilon_r\bm u_h,\nabla p_h)_{\mathcal{T}_h}]\\
	&\ge \frac 1 2 \|\sqrt{{\rm Re}~(\epsilon_r)}~\nabla p_h\|_{\mathcal T_h}^2 -C_3\|\kappa\sqrt{{\rm{Re}}~(\epsilon_r)}~\bm u_h\|_{\mathcal{T}_h}^2,\\
	\vertiii{\bm \tau_3}_h &\le C_4\vertiii{\bm \sigma_h}_h.
	\end{split}
	\end{align}
	Next, , we take $\bm \tau_4 =(\bm 0,\bm 0, \bm 0,\eta_h)$ to get
	\begin{align*}
	\mathscr B_h^+(\bm \sigma_h;\bm \tau_4) = (\epsilon_r\bm u_h,\nabla \eta_h)_{\mathcal T_h}.
	\end{align*}
	By   \eqref{decom1} and the \Cref{vc}, we have 	
	\begin{align}\label{LBB_proof_5}
	\begin{split}
	{\rm{Re}} ~ [  \mathscr B_h^+(\bm \sigma_h;\bm \tau_4)] &=	{\rm{Re}} ~ (\epsilon_r \bm u_h^{\rm c},\nabla \eta_h)_{\mathcal T_h}+	{\rm{Re}} ~ (\epsilon_r(\bm u_h-\bm u_h^{\rm c}),\nabla \eta_h)_{\mathcal T_h}\\
	&=	{\rm{Re}} ~ [(\epsilon_r(\bm y_h+\nabla\eta_h),\nabla\eta_h)_{\mathcal T_h}]+	{\rm{Re}} ~ [(\epsilon_r(\bm u_h-\bm u_h^{\rm c}),\nabla \eta_h)_{\mathcal T_h}]\\
	&={\rm{Re}} ~  [(\epsilon_r \nabla\eta_h,\nabla\eta_h)_{\mathcal T_h}]+ {\rm{Re}} ~  [(\epsilon_r(\bm u_h-\bm u_h^{\rm c}),\nabla \eta_h)_{\mathcal T_h}]\\
	&\ge\frac{1}{2}\|\sqrt{{\rm{Re}}~(\epsilon_r)}~ \nabla\eta_h\|_{\mathcal T_h}^2-C_5\|\bm n\times[\![\bm u_h]\!]\|_{\mathcal F_h}^2\\
	&\ge\frac{1}{2}\|\sqrt{{\rm{Re}}~(\epsilon_r)}~ \nabla\eta_h\|_{\mathcal T_h}^2-C_5\|\bm n\times(\bm u_h-\widehat{\bm u}_h)\|_{\partial\mathcal T_h}^2,\\
	\vertiii{\bm \tau_4}_h &\le C_6\vertiii{\bm \sigma_h}_h.
	\end{split}
	\end{align}
	
	Finally, we take $\bm \tau_h =(1+C_1+C_3+C_5)\bm \tau_1 + \sum_{i=2}^4\bm \tau_i $. By   \eqref{LBB_proof_1} and \eqref{LBB_proof_5} we have
	\begin{align}\label{LBB_proof_6}
	\begin{split}
	{\rm{Re}} ~ [ \mathscr B_h^+(\bm \sigma_h;\bm \tau_h)]\ge C_7 \vertiii{\bm \sigma_h}_h^2,\\
	\vertiii{\bm \tau_h}_h\le   C_8\vertiii{\bm \sigma_h}_h.
	\end{split}
	\end{align}
	The above two inequalities give the desired result \eqref{inf_sup_1}.	
\end{proof}

\section{Error analysis}
\label{Error_estimates}
In this section, we give an error analysis of  the HDG approximation to Maxwell's equations given by \eqref{Maxwell_equation_HDG_form_ori}.  First, we state  main results, i.e.,  \Cref{main_res,real_main_res}. Second, we define a continuous operator $\mathscr A$ and discrete operator $\mathscr A_h$ as in \cite{Buffa_SINUM_2006}, which is a crucial step to get the error estimate without a duality argument. In the end, we provide the proof of \Cref{main_res}.

\subsection{Main result}
Before stating our main result, we introduce the standard $L^2$-orthogonal projection operator $\bm{\Pi}_m^o: \bm L^2(K)\to \bm {\mathcal P}_m(K)$, which satisfies
\begin{align}\label{def_L2}
(\bm \Pi_m^o \bm q, \bm r)_{K}=(\bm q,\bm r)_{K}, \quad \forall \bm r\in \bm {\mathcal{ P}}_m(K).
\end{align}
In the  error analysis, we shall use the following well-known bound:
\begin{align}\label{classical_ine}
\| {\bm q -\bm\Pi_m^o \bm q}\|_{\mathcal T_h} \le Ch^{s} \|{\bm q}\|_{\bm H^s(\Omega)},
\end{align}
where $\bm q \in \bm H^s(\Omega)$ and $s>0$ {is the index} guaranteed by \Cref{final_regu1}. 

Furthermore,  $p\in H^{1+s}(\Omega)$,  which means $p$ may not  continuous on $\Omega$. Therefore, the standard Lagrange interpolation operator is not applicable; hence we utilize the Scott-Zhang interpolation operator $\mathcal I_h: H^{1}(\Omega)\to P_h$.
For the Scott-Zhang interpolant we  have the  following bound bound:
\begin{align}\label{scottzhang2}
\| {\nabla (p -{ \mathcal I_h p})}\|_{\mathcal T_h} \le C  h^{s} \|{p}\|_{H^{1+s}(\Omega)}.
\end{align}

Now we can state the main result:

\begin{theorem}\label{main_res}
	Suppose that  $\mu_r$ and $\epsilon_r$ satisfy  Assumption \ref{Assumptions_coes} and  $\kappa^2$ is not an eigenvalue of  \eqref{maxwell_eig}. Then for $h$ small enough,
	\begin{description}
		\item[(A)] The HDG formulation has a unique solution $ (\bm q_h,\bm u_h, \widehat{\bm u}_h, p_h) \in  \bm Q_h\times\bm U_h\times  \widehat{\bm U}_h\times P_h$;
		
		\item [(B)] Let $(\bm q, \bm u, p)$ be the solution of \eqref{con_pro}, then we have
		\begin{align*}
		&\|\bm q -\bm q_h\|_{\mathcal T_h} + \|\bm u -\bm u_h\|_{\mathcal T_h}+\|\nabla(p - p_h)\|_{\mathcal T_h} \\&\quad \le C\left(\|\bm{\Pi}_m^o \bm q-\bm q\|_{\mathcal T_h} +\|\bm{\mathcal J}_h^{\rm curl}\bm u-\bm u\|_{\mathcal T_h} + \|\bm{\mathcal J}_h^{\rm div} (\bm \nabla\times \bm u)-\bm \nabla\times\bm u\|_{\mathcal T_h}+\|\nabla (\mathcal I_hp - p)\|_{\mathcal T_h}\right),
		\end{align*}
		where $C$ depends  on  $\kappa$, $\Omega$, $\epsilon_r$ and $\mu_r$. Here $\bm{\Pi}_m^o$ denotes $L^2$ projection (see \eqref{def_L2}).
		
	\end{description}	
\end{theorem}

\begin{remark}
	There exist many papers utilizing variants of HDG method to discretize Maxwell equations, for example,  \cite{Fen_CMAM_2016,Chen_CMAME_2018,Lu_Mathcomp_2017,Christophe_AMC_2018,Chen_maxwell_HDG_2017}. However, to the best of our knowledge, these papers do not allow piecewise smooth coefficients. Our main result \Cref{main_res} relaxes the requirements on the coefficients.  We note that a similar result has also been recently obtained in \cite[Theorem 3.3]{Ern_CMAM_2018}  for conforming edge elements.
	
	By the approximation {properties of} $\bm{\mathcal J}_h^{\rm div}$, $\bm{\Pi}_m^o$ and  $\mathcal I_h$ in \eqref{app_div}, \eqref{classical_ine} and \eqref{scottzhang2} and using the regularity result in \Cref{final_regu1}, we have:
	\begin{corollary}\label{real_main_res} 
		For general coefficients satisfying Assumption \ref{Assumptions_coes}, and assuming $\kappa^2$ is not a Maxwell eigenvalue, then 
		\begin{align*}
		\|\bm q -\bm q_h\|_{\mathcal T_h} + \|\bm u -\bm u_h\|_{\mathcal T_h}+\|\nabla(p - p_h)\|_{\mathcal T_h} \le Ch^s(\|\bm f\|_{\bm L^2(\Omega)}+\|\bm \nabla\cdot\bm f\|_{L^2(\Omega)}),
		\end{align*}
		where  $s>0$ {and $C$ depends  on $s$, $\kappa$,} $\Omega$, $\epsilon_r$ and $\mu_r$.
	\end{corollary}

\end{remark}

\subsection{Preliminary {estimates}}

\begin{definition}\label{def_A}
	Let $\bm g \in \bm L^2(\Omega)$, we define the operator ${\mathscr  A}:=({\mathscr A}^{\bm q},{\mathscr A}^{\bm u}, {\mathscr  A}^p): \bm L^2(\Omega) \to \bm H({\rm curl};\Omega)\times \bm H_0({\rm curl};\Omega)\times  H^1_0(\Omega)$ such that for all  $\bm \tau = (\bm r,\bm v,\chi)\in \bm H({\rm curl};\Omega)\times \bm H_0({\rm curl};\Omega)\times H^1_0(\Omega)$ we have 
	\begin{align}\label{def_Ag}
	{\mathscr B}^+({\mathscr A}\bm g; \bm \tau)	=(\epsilon_r \bm g,\bm v)_{\bm L^2(\Omega)}.
	\end{align}
\end{definition}

By the \Cref{def_A} and \Cref{Th2} we get
\begin{align}\label{reg_A_u}
\|{\mathscr A}^{\bm q} \bm g\|_{\bm H^s(\Omega)} + \|{\mathscr A}^{\bm u}\bm g\|_{\bm H^s(\Omega)} \le C\|\epsilon_r \bm g\|_{L^2(\Omega)}.
\end{align}  
Moreover, if $\bm g \in \bm H({\rm div}_{\epsilon_r}, \Omega)$, then we have 
\begin{align}\label{reg_A_p}
\|{\mathscr A}^{p} \bm g\|_{H^{1+s}(\Omega)} \le C\|\bm\nabla  \cdot (\epsilon_r \bm g)\|_{L^2(\Omega)}.
\end{align}

\begin{definition}\label{defahg}
	Let $\bm g\in \bm L^2(\Omega)$, we define the operator $\mathscr {A}_h:=({\mathscr A}_h^{\bm q},{\mathscr A}_h^{\bm u},{\mathscr A}_h^{\widehat {\bm u}}, {\mathscr A}_h^p): \bm L^2(\Omega)\to\bm Q_h\times\bm U_h\times  \widehat{\bm U}_h\times P_h$ such that
	\begin{align}\label{def_Ahg}
	{\mathscr  B}_h^+(\mathscr{A}_h\bm g;\bm \tau_h) =(\epsilon_r \bm g,\bm v_h)_{\mathcal T_h}
	\end{align} 
	for all $\bm \tau_h = (\bm r_h,\bm v_h, \widehat{\bm v}_h,\chi_h)$.
	Moreover, we define  $\bm g_h\in \bm U_h$ {to be the unique solution of}
	\begin{align*}
	(\epsilon_r \bm g_h,\bm v_h)_{\mathcal T_h}=(\bm g,\bm v_h)_{\mathcal T_h}.
	\end{align*}
\end{definition}

\begin{lemma}[Stability for $\mathscr A$ and $\mathscr A_h$] \label{stability_A}
	Let $\bm g\in \bm L^2(\Omega)$, then we have 
	\begin{subequations}
		\begin{align}
		\vertiii{\mathscr A{\bm g} }\le C\|\bm g\|_{\bm L^2(\Omega)},\label{stability_A1}\\
		\vertiii{\mathscr A_h{\bm g} }_h\le C\|\bm g\|_{\bm L^2(\Omega)}.\label{stability_A2}
		\end{align}
	\end{subequations}
	where $\vertiii{\cdot}$ and $\vertiii{\cdot}_h$ were defined in \eqref{trip_morm_defc},\eqref{trip_morm_def}, respectively.
\end{lemma}
\begin{proof}
	Let  $ \bm \tau = (\bm r, \bm v, \chi)\in \bm L^2(\Omega)\times  \bm H_0({\rm curl};\Omega)\times H^1_0(\Omega)$, by \Cref{LBBc} and  \Cref{def_A} we get
	\begin{align*}
	\vertiii{\mathscr A{\bm g} } \le C	\sup_{ \bm \tau \neq \bm 0}\frac{{\rm{Re}}~ [\mathscr B^+ (\mathscr A{\bm g};\bm \tau)]}{\vertiii{\bm \tau}}=C\sup_{\bm \tau \neq \bm 0 }\frac{{\rm{Re}}~[(\epsilon_r \bm g,\bm v)_{\bm L^2(\Omega)}]}{\vertiii{\bm \tau }}\le C\|\bm g\|_{\bm L^2(\Omega)}.
	\end{align*}
	Next, let $\bm \tau_h = (\bm r_h,\bm v_h, \widehat{\bm v}_h,\chi_h) \in \bm Q_h\times\bm U_h\times  \widehat{\bm U}_h\times P_h$. We use   \Cref{LBB} and \Cref{def_A} to get
	\begin{align*}
	\vertiii{\mathscr A_h{\bm g} }_h\le C	\sup_{\bm \tau_h \neq \bm 0 }\frac{ {\rm{Re}} ~ [\mathscr B_h^+ (\mathscr A_h{\bm g}; \bm \tau_h)]}{\vertiii{\bm \tau_h}_{h}}= C	\sup_{\bm \tau_h \neq \bm 0 }\frac{ {\rm{Re}} ~ [ (\epsilon_r\bm g;\bm v_h)_{\mathcal T_h}]}{\vertiii{\bm \tau_h}_{h}}\le C\|\bm g\|_{\bm L^2(\Omega)}.
	\end{align*}
\end{proof}

Next, we define 
\begin{align}\label{widetilde_A_g}
\widetilde{\mathscr A \bm g} = ({\mathscr A}^{\bm q} \bm g,{\mathscr A}^{\bm u} \bm g,{\mathscr A}^{\bm u} \bm g, {\mathscr  A}^p \bm g),
\end{align}
where $\mathscr A \bm g$ was defined in \eqref{def_Ag}. Then for all $\bm \tau_h = (\bm r_h, \bm v_h, \widehat{\bm v}_h, \chi_h)\in  \bm Q_h\times\bm U_h\times  \widehat{\bm U}_h\times P_h$,  by the definition of $\mathscr B_h^+$ in \eqref{def_Bh},  \eqref{def_Ag} and the definition of $\mathscr B^+$ in \eqref{def_Bplus} we have 
\begin{align}\label{important_rela}
\mathscr B_h^+(\widetilde{\mathscr A}\bm g;\bm \tau_h) = (\epsilon_r\bm g, \bm v_h)_{\mathcal T_h}={\mathscr  B}_h^+(\mathscr{A}_h\bm g;\bm \tau_h).
\end{align}

\begin{lemma} 
	For any $\bm g \in \bm H({\rm div}_{\epsilon_r};\Omega)$, we have
	\begin{align}\label{estAuAuh}
	\vertiii{\widetilde{\mathscr A} \bm g -\mathscr A_h \bm g}_h \le Ch^s(\|\bm g\|_{\bm L^2(\Omega)}+\|\bm \nabla\cdot(\epsilon_r\bm g)\|_{L^2(\Omega)}).
	\end{align}
\end{lemma}

\begin{proof}
	First, let  $ \mathscr I_h \left(\widetilde{\mathscr A} \bm g\right) = (\bm{\Pi}_m^o(\mathscr A^{\bm q}\bm g), \bm{\mathcal J}_h^{\rm curl}(\mathscr A^{\bm u}\bm g), \bm n\times(\bm{\mathcal J}_h^{\rm curl} (\mathscr A^{\bm u}\bm g))\times\bm n,\mathcal I_h(\mathscr A^p\bm g))\in \bm Q_h\times\bm U_h\times  \widehat{\bm U}_h\times P_h $, where $\bm{\Pi}_m^o$ was  defined in \eqref{def_L2} and  $\mathcal I_h$ is the Scott-Zhang interpolant.   Then 
	\begin{align}
	\hspace{1em}&\hspace{-1em}\vertiii{\mathscr I_h \left(\widetilde{\mathscr A} \bm g\right) -\mathscr A_h \bm g}_h\nonumber\\
	& \le C\sup_{\bm 0\neq \bm{\tau}_h\in\bm{\Sigma}_h}\frac{{\rm Re}~\left[\mathscr  B_h^+\left(\mathscr I_h \left(\widetilde{\mathscr A} \bm g\right) -\mathscr A_h\bm g;\bm{\tau}_h\right)\right]}{\vertiii{\bm{\tau}_h}_h}&\textup{ by }\eqref{inf_sup_1}\nonumber\\
	&\le C	\sup_{\bm 0\neq \bm{\tau}_h\in\bm{\Sigma}_h}\frac{{\rm Re}~ \left[\mathscr  B_h^+\left(\mathscr I_h \left(\widetilde{\mathscr A} \bm g\right)-\widetilde{\mathscr A}\bm g;\bm{\tau}_h\right)\right]}{\vertiii{\bm{\tau}_h}_h}&\textup{ by }\eqref{important_rela}\nonumber\\
	&\le C\vertiii{\mathscr I_h \left(\widetilde{\mathscr A} \bm g\right) - \widetilde{\mathscr A}\bm g}_h&\textup{ by }\eqref{boundnessofBminus}. \label{bounded_inter}
	\end{align}
	Next by the triangle inequality we have 
	\begin{align*}
	\hspace{1em}&\hspace{-1em}\vertiii{\widetilde{\mathscr A} \bm g -\mathscr A_h \bm g}_h \\
	&\le \vertiii{ \widetilde{\mathscr A}\bm g-\mathscr I_h \left(\widetilde{\mathscr A} \bm g\right) }_h  + \vertiii{\mathscr I_h \left(\widetilde{\mathscr A} \bm g\right)-\mathscr A_h \bm g}_h \\
	&\le C \vertiii{ \widetilde{\mathscr A}\bm g-\mathscr I_h \left(\widetilde{\mathscr A} \bm g\right) }_h &\textup{ by }\eqref{bounded_inter}\\
	& \le C\left(\|\bm{\Pi}_m^o(\mathscr A^{\bm q}\bm g)-\mathscr A^{\bm q}\bm g\|_{\mathcal T_h} +\|\bm{\mathcal J}_h^{\rm curl}(\mathscr A^{\bm u}\bm g)-\mathscr A^{\bm u}\bm g\|_{\mathcal T_h}\right.\\
	&\quad + \left.\|\bm \nabla\times (\bm{\mathcal J}_h^{\rm curl}(\mathscr A^{\bm u}\bm g)-\mathscr A^{\bm u}\bm g)\|_{\mathcal T_h}+\| \nabla (\mathcal I_h(\mathscr A^p\bm g)-\mathscr A^p\bm g)\|_{\mathcal T_h}\right)&\textup{ by \eqref{boundnessofBminus}}\\
	&=C\left(\|\bm{\Pi}_m^o(\mathscr A^{\bm q}\bm g)-\mathscr A^{\bm q}\bm g\|_{\mathcal T_h} +\|\bm{\mathcal J}_h^{\rm curl}(\mathscr A^{\bm u}\bm g)-\mathscr A^{\bm u}\bm g\|_{\mathcal T_h}\right.\\
	&\quad + \left.\|\bm{\mathcal J}_h^{\rm div} (\bm \nabla\times(\mathscr A^{\bm u}\bm g))-\bm \nabla\times(\mathscr A^{\bm u}\bm g)\|_{\mathcal T_h}+\|\nabla (\mathcal I_h(\mathscr A^p\bm g)-\mathscr A^p\bm g)\|_{\mathcal T_h}\right)&\textup{ by \eqref{curl_div_commute}}.
	\end{align*}
	By the approximation {properties of}  $\bm{\Pi}_m^o$, $\mathcal I_h$ and $\bm{\mathcal J}_h^{\rm div}$ in \eqref{classical_ine}, \eqref{scottzhang2} and \eqref{app_div} we get
	\begin{align*}
	\hspace{1em}&\hspace{-1em}\vertiii{\widetilde{\mathscr A} \bm g -\mathscr A_h \bm g}_h\\
	& \le Ch^s (\|\mathscr A^{\bm q}\bm g\|_{\bm H^s(\Omega)} + \|\mathscr A^{\bm u}\bm g\|_{\bm H^s(\Omega)}+\|\bm \nabla \times (\mathscr A^{\bm u}\bm g)\|_{\bm H^s(\Omega)} + \|\mathscr A^p\bm g\|_{H^{1+s}(\Omega)})\\
	&\le Ch^s (\|\bm g\|_{\bm L^2(\Omega)} + \|\bm \nabla\cdot(\epsilon_r\bm g)\|_{L^2(\Omega)}),
	\end{align*}
	where we used the regularity results \eqref{reg_A_u} and \eqref{reg_A_p}. 
\end{proof}

Next, we define the following norm on the space $\bm W = \bm H_0({\rm curl}, \Omega)+\bm U_h$ by 
\begin{align}\label{def_norm_Uh}
\|\bm w\|_{\bm W}^2:= \|\sqrt{\rm Re~{(\epsilon_r)}}~\bm w\|^2_{\mathcal T_h}+
\|\bm \nabla\times\bm w\|_{\mathcal T_h}^2+\|{\bf h}^{-\frac{1}{2}}\bm n\times[\![\bm w]\!]\|_{\mathcal F_h}^2.
\end{align}

\begin{lemma}\label{est_auahubound}
	For any $\bm u_h\in \bm U_h$, we have
	\begin{align*}
	\|\mathscr A^{\bm u} \bm u_h-\mathscr A_h^{\bm u}\bm u_h\|_{\bm W}\le Ch^{s}\|\bm u_h\|_{\bm W}.
	\end{align*}
\end{lemma}

\begin{proof} 
	Let $\bm u_h^{\rm c}$ be defined as in \Cref{vc}, then by \Cref{Helmholtz_discrete} we have 
	\begin{align}\label{def_proof0}
	\bm u_h^{\rm c}=\bm z_h+\nabla\xi_h,\quad 	(\epsilon_r \bm z_h,\nabla\eta_h)_{\mathcal T_h}=0,
	\end{align}
	where $\eta_h \in P_h$ and 
	\begin{align*}
	\|\nabla\xi_h\|_{\mathcal T_h}\le C\|\bm u_h^{\rm c}\|_{\mathcal T_h},\quad 
	\|\bm z_h\|_{\mathcal T_h}\le C\|\bm u_h^{\rm c}\|_{\mathcal T_h}.
	\end{align*}

	Let $\bm{\Theta}\in \bm H_0({\rm curl};\Omega)\cap \bm H({\rm div}_{\epsilon_r}^0;\Omega)$  be the solution of
	\begin{align*}
	\bm \nabla\times\bm{\Theta}=\bm \nabla\times\bm z_h.
	\end{align*}
	Then by  \eqref{theta_bound} in \Cref{H_curl_app} we have
	\begin{align}\label{def_proof1}
	\|\bm{\Theta}-\bm z_h\|_{\mathcal T_h}\le Ch^{s}\|\bm\nabla\times\bm z_h\|_{\mathcal T_h}.
	\end{align}
	Moreover, for all $\xi_h \in P_h$ we have  $(\epsilon_r \bm u_h, \nabla \xi_h)_{\mathcal T_h}=0$  and  $(\epsilon_r\bm \Theta,  \nabla \xi_h)_{\mathcal T_h}=0$, then
	\begin{align*}
	\|\bm \Theta - \bm u_h\|_{\mathcal T_h}^2 &\le C {\rm Re}~ [(\epsilon_r(\bm \Theta -\bm u_h), \bm \Theta -\bm u_h)_{\mathcal T_h}]\\
	&\le C {\rm Re}~ [(\epsilon_r(\bm \Theta -\bm u_h), \bm \Theta -\bm z_h -\nabla \xi_h +\bm u_h^{\rm c} - \bm u_h)_{\mathcal T_h}]\\
	&\le C {\rm Re}~ [(\epsilon_r(\bm \Theta -\bm u_h), \bm \Theta -\bm z_h +\bm u_h^{\rm c} - \bm u_h)_{\mathcal T_h}].
	\end{align*}
	This gives
	\begin{align}\label{def_proof2}
	\|\bm \Theta - \bm u_h\|_{\mathcal T_h} \le C  (\|\bm \Theta -\bm z_h\|_{\mathcal T_h} + \|\bm u_h^{\rm c} - \bm u_h\|_{\mathcal T_h}).
	\end{align}
	By the definitions of $\|\cdot\|_{\bm W}$ and $\vertiii{\cdot}_h $ in  \eqref{def_norm_Uh} and \eqref{trip_morm_def}, we have,
	{for $h$ small enough and $s\leq 1/2$,}
	\begin{align*}
	\hspace{1em}&\hspace{-1em} \|(\mathscr A^{\bm u}-\mathscr A_h^{\bm u})\bm u_h\|_{\bm W}\\
	&\le	\vertiii{(\widetilde{\mathscr A} -\mathscr A_h) \bm u_h}_h&\textup{ by }\eqref{widetilde_A_g}\\
	&\le \vertiii{(\widetilde{\mathscr A}-\mathscr A_h)(\bm u_h-\bm{\Theta})}_h
	+\vertiii{(\widetilde{\mathscr A}-\mathscr A_h)\bm{\Theta}}_{h}\\
	&\le C\left(\|\bm u_h -\bm{\Theta}\|_{\bm L^2(\Omega)}
	+h^s\|\bm{\Theta}\|_{\bm L^2(\Omega)}\right)&\textup{ by }\eqref{stability_A1} \textup{ and } \eqref{estAuAuh} \\
	&\le C\left(\|\bm u_h -\bm{\Theta}\|_{\bm L^2(\Omega)}
	+h^s( \|\bm u_h-\bm u_h^{\rm c}\|_{\bm L^2(\Omega)} + \|\bm u_h^{\rm c}\|_{\bm L^2(\Omega)})\right)\\
	&\le C\left(\|\bm \Theta -\bm z_h\|_{\mathcal T_h} + \|\bm u_h^{\rm c} - \bm u_h\|_{\mathcal T_h} + h^s\|\bm u_h^{\rm c}\|_{\bm L^2(\Omega)}\right)&\textup{ by }\eqref{def_proof2}\\
	&\le Ch^s\left(\|\bm \nabla \times \bm z_h\|_{\bm L^2(\Omega)}
	+\|{\bf h}^{{-\frac{1}{2}}}\bm n\times[\![\bm u_h]\!]\|_{\mathcal F_h}+\|\bm u_h^{\rm c} \|_{\bm L^2(\Omega)}\right)&\textup{ by }\eqref{def_proof1}\textup{ and }\eqref{est_vc}\\
	&\le Ch^s\left(\|{\bf h}^{{-\frac{1}{2}}}\bm n\times[\![\bm u_h]\!]\|_{\mathcal F_h}+\|\bm\nabla\times\bm u_h^{\rm c}\|_{\mathcal F_h}+\|\bm u_h^{\rm c}\|_{\bm L^2(\Omega)}\right)&\textup{ by }\eqref{def_proof0}\\
	&\le Ch^s\|\bm u_h\|_{\bm W}&\textup{ by }\eqref{est_vc}.
	\end{align*}
\end{proof}

\begin{lemma}\label{est_eig_pp}
	If $\kappa^2$ is not an eigenvalue of the problem \eqref{maxwell_eig}, then for all $\bm w \in \bm W = \bm U_h+\bm H_0({\rm curl}, \Omega)$,  there exists a positive constant $C$ only depending on $\Omega$ and $\kappa$ such that,
	\begin{align*}
	\left\|{\frac{\bm w}{1+\kappa^2} -\mathscr A^{\bm u}\bm w}\right\|_{\bm W}\ge C\|\bm w\|_{\bm W}.
	\end{align*}
\end{lemma}
\begin{proof}
	Let $z = 1/(1+\kappa^2)$ and $\bm g=(z-\mathscr A^{\bm u})\bm w$, then {
		$\bm g\in \bm H_0({\rm curl}, \Omega)+\bm U_h$.  This implies $z\bm w-\bm g =\mathscr A^{\bm u}\bm w  \in \bm H_0({\rm curl}, \Omega)$. By the definition of $\mathscr A^{\bm u}$ {(see \eqref{def_A})}, we know  $\mathscr A^{\bm u}\bm w$ satisfies the following equation:}
	\begin{align*}
	\bm \nabla\times(\mu_r^{-1}\bm \nabla\times(\mathscr A^{\bm u}\bm w))+\epsilon_r(\mathscr A^{\bm u}\bm w) +\bar\epsilon_r \mathscr A^p \bm w=  \epsilon_r\bm w. 
	\end{align*}
	{Setting $\bm w=(1/z)(\mathscr A^{\bm u}\bm w+\bm g)$ on the right hand side gives
		\begin{align*}
		\bm \nabla\times(\mu_r^{-1}\bm \nabla\times(\mathscr A^{\bm u}\bm w))-\kappa^2\epsilon_r(\mathscr A^{\bm u}\bm w) +\bar\epsilon_r \mathscr A^p \bm w=  \frac{1}{z}\epsilon_r\bm g,
		\end{align*}	
		and so $z\bm w-\bm g\in \bm H_0({\rm curl}, \Omega)$} satisfies the following equation:
	\begin{align}
	\bm \nabla\times(\mu_r^{-1}\bm \nabla\times((z\bm w-\bm g)))-\kappa^2\epsilon_r(z\bm w-\bm g) +\bar\epsilon_r \mathscr A^p \bm w =\frac{1}{z}\epsilon_r\bm g.
	\end{align}
	Since $\kappa^2$ is not an eigenvalue of problem \eqref{maxwell_eig}, then by the \Cref{final_regu1} we have 
	\begin{align}\label{inver_proof_1}
	\|z\bm w-\bm g\|_{\bm H({\rm curl}, \Omega)}\le
	\frac{C}{|z|}\|\epsilon_r\bm g\|_{\bm L^2(\Omega)}\le \frac{C}{|z|}\|\bm g\|_{\bm W}.
	\end{align}
	Since $	\|z\bm w-\bm g\|_{\bm H({\rm curl}, \Omega)} = \|z\bm w-\bm g\|_{\bm W}$, then by   \eqref{inver_proof_1} we have 
	\begin{align*}
	\|\bm w\|_{\bm W}\le \frac{1}{|z|}(\|z\bm w-\bm g\|_{\bm W} + \|\bm g\|_{\bm W}) \le C\|\bm g\|_{\bm W}=	C\|(z-\mathscr A^{\bm u})\bm w\|_{\bm W}.
	\end{align*}
\end{proof}

\begin{lemma}\label{est}
	If $\kappa^2$ is not an eigenvalue of problem \eqref{maxwell_eig} and $h$ is small enough, then for all $\bm w \in \bm W = \bm U_h+\bm H_0({\rm curl}, \Omega)$,  there exists a positive constant $C$ only depending on $\Omega$ and $\kappa$ such that,
	\begin{align*}
	\left\|{\frac{\bm w}{1+\kappa^2} -\mathscr A_h^{\bm u}\bm w}\right\|_{\bm W}\ge C\|\bm w\|_{\bm W}.
	\end{align*}
\end{lemma}
\begin{proof} 
	Let $z = 1/(1+\kappa^2)$, by  \Cref{est_auahubound,est_eig_pp} and the triangle inequality we have
	\begin{align*}
	\|(z-\mathscr A_h^{\bm u})\bm w\|_{\bm W}\ge 
	\|(z-\mathscr A^{\bm u})\bm w\|_{\bm W}
	-\|(\mathscr A^{\bm u}-\mathscr A^{\bm u}_h)\bm w\|_{\bm W}\ge (C-h^s)\|\bm w\|_{\bm W}.
	\end{align*} 
	The desired result holds if $ h $ small enough.
\end{proof}
The next result follows from the coercivity proved in the previous lemma.
\begin{corollary}
	If $\kappa^2$ is not an eigenvalue of problem \eqref{maxwell_eig} and $h$ is small enough, then for all $\bm w_h \in \bm U_h$,  there exists a positive constant $C$ only depending on $\Omega$ and $\kappa$ such that,
	\begin{align}\label{rest}
	\left\|{\left(\frac{1}{1+\kappa^2}-\mathscr A_h^{\bm u}\right)^{-1}\bm w_h}\right\|_{\bm  W}\le C\|\bm w_h\|_{\bm W}.
	\end{align}
\end{corollary}

\subsection{Proof of (A) in \Cref{main_res}}

\begin{lemma}\label{uniqueness_uh}
	If $\kappa^2$ is not an eigenvalue of problem \eqref{maxwell_eig} and $h$ is small enough,   then the HDG scheme \eqref{con_pro} has a unique solution $\bm{\sigma}_h =  (\bm q_h,\bm u_h, \widehat{\bm u}_h, p_h)\in \bm  Q_h\times\bm U_h\times  \widehat{\bm U}_h\times P_h$. Moreover, we have 
	\begin{align}\label{uniqueness_uhbound}
	\vertiii{\bm \sigma_h}_h\le C\|\bm f\|_{\bm L^2(\Omega)}.
	\end{align}
\end{lemma}
\begin{proof} 
	{Suppose the solution $\bm{\sigma}_h$ exists. Then for}
	all $\bm{\tau}_h = (\bm r_h,\bm v_h, \widehat{\bm v}_h,\chi_h)\in \bm Q_h\times\bm U_h\times  \widehat{\bm U}_h\times P_h$ we have
	\begin{align}\label{uniqueness_uh_proof0}
	\mathscr  B_h^{-}(\bm{\sigma}_h;\bm{\tau}_h)=\mathscr B_h^{+}(\bm{\sigma}_h;\bm{\tau}_h)-(1+\kappa^2)
	(\epsilon_r \bm u_h,\bm v_h)_{\mathcal T_h}=(\bm f,\bm v_h)_{\mathcal T_h}.
	\end{align}
	i.e., 
	\begin{align}\label{uniqueness_uh_proof1}
	\frac{1}{1+\kappa^2} \mathscr B_h^{+}(\bm{\sigma}_h;\bm{\tau}_h)-
	(\epsilon_r \bm u_h,\bm v_h)_{\mathcal T_h}=\frac{1}{1+\kappa^2}(\bm f,\bm v_h)_{\mathcal T_h}.
	\end{align}
	By the \Cref{defahg} we have
	\begin{subequations} \label{uniqueness_uh_proof2}
		\begin{align}
		{\mathscr  B}_h^+(\mathscr{A}_h\bm u_h;\bm \tau_h) &=(\epsilon_r \bm u_h,\bm v_h)_{\mathcal T_h},\label{uniqueness_uh_proof21}\\
		\frac{1}{1+\kappa^2}(\bm f,\bm v_h)_{\mathcal T_h} &= \frac{1}{1+\kappa^2}(\epsilon_r \bm f_h,\bm v_h)_{\mathcal T_h}=\frac{1}{1+\kappa^2}{\mathscr  B}_h^+(\mathscr{A}_h\bm f_h;\bm \tau_h).\label{uniqueness_uh_proof22}
		\end{align}
	\end{subequations} 
	By   \eqref{uniqueness_uh_proof2},  we can rewrite  \eqref{uniqueness_uh_proof1} as follows
	\begin{align*}
	\mathscr B_h^+\left(\frac{1}{1+\kappa^2}\bm\sigma_h-\mathscr A_h \bm u_h-\frac{1}{1+\kappa^2}\mathscr A_h\bm f_h, \bm{\tau}_h\right)=0.
	\end{align*}
	Due to the coercivity of $\mathscr B_h^+$,  we have
	\begin{align*}
	\frac{1}{1+\kappa^2}\bm\sigma_h-\mathscr A_h \bm u_h-\frac{1}{1+\kappa^2} \mathscr A_h\bm f_h=\bm 0,
	\end{align*}
	i.e., 
	\begin{align}\label{uniqueness_uh_proof3}
	\left(\frac{1}{1+\kappa^2}-\mathscr A_h^{\bm u}\right)\bm u_h=\frac{1}{1+\kappa^2}\mathscr A_h^{\bm u}\bm f_h,
	\end{align}
	which is uniquely solvable by \Cref{est}. Moreover, by \eqref{rest}, we have 
	\begin{align}
	\|\bm u_h\|_{\mathcal T_h}&\le C\|\bm u_h\|_{\bm W}&\textup{ by }\eqref{def_norm_Uh}\nonumber\\
	&\le C\|\mathscr A_h^{\bm u}\bm f_h\|_{\bm W}&\textup{ by }\eqref{rest} \textup{ and }\eqref{uniqueness_uh_proof3}\nonumber\\
	&\le C\vertiii{\mathscr A_h{\bm f}_h }_h&\textup{ by }\eqref{trip_morm_def}\nonumber\\
	&\le C\|\bm f_h\|_{\bm L^2(\Omega)} &\textup{ by }\eqref{stability_A2}\nonumber\\
	&\le C\|\bm f\|_{\bm L^2(\Omega)}&\textup{ by }\eqref{uniqueness_uh_proof22}. \label{uniqueness_uh_proof4}
	\end{align}

	Next,  by  \eqref{uniqueness_uh_proof0} we get 
	\begin{align}\label{uniqueness_uh_proof5}
	\mathscr B_h^+(\bm{\sigma}_h;\bm{\tau}_h)=(1+\kappa^2)(\epsilon_r\bm u_h,\bm v_h)_{\mathcal T_h}+(\bm f,\bm v_h)_{\mathcal T_h}
	\end{align}
	for all $\bm{\tau}_h = (\bm r_h,\bm v_h, \widehat{\bm v}_h,\chi_h)\in \bm Q_h\times\bm U_h\times \widehat{\bm U}_h\times P_h$. 
	{We can now prove the uniqueness of any solution to the discrete HDG problem. Suppose $\bm f=0$, then $\bm u_h=0$ by \eqref{uniqueness_uh_proof4}, and so $\bm{\sigma}_h=0$ by (\ref{uniqueness_uh_proof5}) and  \Cref{LBB}.  Since the linear system corresponding to the discrete HDG problem is
		square, uniqueness implies existence.} Therefore, the HDG scheme \eqref{con_pro} has a unique solution. 
	Furthermore, 
	\begin{align*}
	\vertiii{\bm \sigma_h }_h&\le C	\sup_{\bm \tau_h \neq \bm 0 }\frac{ {\rm{Re}} ~ [\mathscr B_h^+ (\bm \sigma_h; \bm \tau_h)]}{\vertiii{\bm \tau_h}_{h}}&\textup{ by }\eqref{inf_sup_1}\\
	&= C	\sup_{\bm \tau_h \neq \bm 0 }\frac{ {\rm{Re}} ~ [(1+\kappa^2)(\epsilon_r\bm u_h,\bm v_h)_{\mathcal T_h}+(\bm f,\bm v_h)_{\mathcal T_h}]}{\vertiii{\bm \tau_h}_{h}}&\textup{ by }\eqref{uniqueness_uh_proof5}\\
	&\le C(\|\bm u_h\|_{\bm L^2(\Omega)} + \|\bm f\|_{\bm L^2(\Omega)})\\
	&\le C\|\bm f\|_{\bm L^2(\Omega)} &\textup{ by }\eqref{uniqueness_uh_proof4}.
	\end{align*}
\end{proof}

\subsection{Proof of (B) in \Cref{main_res}}
{We first prove that a discrete inf-sup condition holds.}
\begin{lemma}[Discrete inf-sup conditions on $\mathscr B_h^-$] 
	Let $\bm{\sigma}_h=(\bm q_h,\bm u_h, \widehat{\bm u}_h, p_h), \bm \tau_h=(\bm r_h,\bm v_h, \widehat{\bm v}_h,\\ \chi_h) \in \bm{\Sigma}_h=\bm Q_h\times\bm U_h\times \widehat{\bm U}_h\times P_h$, {and suppose $\kappa^2$ is not a Maxwell eigenvalue and } $h$ small enough.  Then we have  following inf-sup condition
	\begin{align}
	\sup_{\bm 0\neq\bm{\sigma}_h\in\bm{\Sigma}_h}\frac{{\rm Re}~[\mathscr B_h^-(\bm{\sigma}_h;\bm{\tau}_h)]}{\vertiii{\bm{\tau}_h}_{h}}\ge C\vertiii{\bm{\sigma}_h}_{h}.\label{dis_inf_sup_1}
	\end{align}
\end{lemma}
\begin{proof} 
	By   \eqref{LBB_proof_6} in {the proof of} \Cref{LBB}, there exists  $\bm{\tau}_1=(\bm r_1,\bm v_1,\widehat{\bm v}_1, \chi_1)\in \bm{\Sigma}_h$ such that
	\begin{align*}
	{\rm{Re}} ~ [ \mathscr B_h^+(\bm \sigma_h;\bm \tau_1)]\ge C_1 \vertiii{\bm \sigma_h}_h^2,\\
	\vertiii{\bm \tau_1}_h\le   C_2\vertiii{\bm \sigma_h}_h.
	\end{align*}
	This is equivalent to 
	\begin{align*}
	{\rm Re}~[\mathscr B_h^-(\bm{\sigma}_h;\bm{\tau}_1) + (1+\kappa^2)(\epsilon_r \bm u_h,\bm v_1)_{\mathcal T_h}
	]\ge C_1\vertiii{\bm{\sigma}_h}_{h}^2.
	\end{align*}
	By the Cauchy-Schwarz inequality we have 
	\begin{align}\label{def_proof00}
	{\rm Re}~[\mathscr B_h^-(\bm{\sigma}_h;\bm{\tau}_1) 
	]\ge C_3\vertiii{\bm{\sigma}_h}_{h}^2 - C_4(1+\kappa^2)\|\epsilon_r\|_{L^\infty(\Omega)}\| \bm u_h\|_{\mathcal T_h}^2.
	\end{align}

	For  $\bm u_h\in\bm U_h$, {we may choose $\bm u_h^{\rm c} \in \bm U_h\cap \bm H_0({\rm curl};\Omega)$ so that the estimate in \Cref{vc}, is satisfied.}  Then by \Cref{Helmholtz_discrete}, there exist $\bm z_h\in \bm U_h\cap \bm H_0({\rm curl};\Omega)$ and $\xi_h\in P_h$, such that for all $\chi_h\in P_h$ we have 
	\begin{align}\label{decom2}
	\bm u_h^{\rm c}=\bm z_h+\nabla \xi_h, \quad (\epsilon_r \bm z_h,\nabla \chi_h)_{\mathcal T_h}=0.
	\end{align}
	Let $\bm{\Theta}\in \bm H_0({\rm curl};\Omega)\cap \bm H({\rm div}^0_{\epsilon_r};\Omega)$  be the solution of
	\begin{align*}
	\bm\nabla\times\bm{\Theta}=\bm\nabla\times \bm z_h.
	\end{align*}
	Then by   \eqref{theta_bound}  we have
	\begin{align}\label{def_proof3}
	\|\bm{\Theta}-\bm z_h\|_{\mathcal T_h} &\le Ch^{s}\|\bm\nabla\times\bm z_h\|_{\mathcal T_h} \nonumber\\
	&= Ch^{s}\|\bm\nabla\times\bm u_h^{\rm c}\|_{\mathcal T_h} &\textup{ by } \eqref{decom2}\nonumber\\
	&\le  Ch^{s}(\|\bm\nabla\times(\bm u_h^{\rm c} - \bm u_h)\|_{\mathcal T_h} + \|\bm\nabla\times  \bm u_h\|_{\mathcal T_h}) \\
	& \le  Ch^{s}(\|{\bf h}^{\frac{1}{2}}\bm n\times[\![\bm u_h]\!]\|_{\mathcal F_h}+ \|\bm\nabla\times  \bm u_h\|_{\mathcal T_h})&\textup{ by } \rm{\Cref{vc}}\nonumber.
	\end{align}
	Let $\bm{\tau}_2=(\bm r_2,\bm v_2,\widehat{\bm v}_2,\chi_2)\in\bm{\Sigma}_h$ be the solution of
	\begin{align*}
	\mathscr B_h^-(\bm{\tau}_2;\bm{\tau}_h)=(\bm{\Theta},{\epsilon}{_r}\bm v_h)_{\mathcal T_h}
	\end{align*}
	guaranteed by \Cref{uniqueness_uh}.
	This implies
	\begin{align}\label{error_proof1}
	\mathscr B_h^-(\bm{\tau}_2;\bm{\sigma}_h)=(\bm{\Theta},\epsilon{_r}\bm u_h)_{\mathcal T_h}.
	\end{align}
	Furthermore, by \eqref{uniqueness_uhbound} we have 
	\begin{align}\label{error_proof2}
	\vertiii{\bm{\tau}_2}_{h}\le C\|\bm{\Theta}\|_{\mathcal T_h}.
	\end{align}
	
	Next, we take $\bm{\tau}_2^\star=(-\bm r_2,\bm v_2,\widehat{\bm v}_2,-\chi_2)$ and
	$\bm{\tau}_3=(-\bm q_h,\bm u_h,\widehat{\bm u}_h,-p_h)$ to get
	\begin{align*}
	\mathscr B_h^-(\bm{\sigma}_h,\bm{\tau}_2^\star)&=
	\overline{\mathscr B_h^-(\bm{\tau}_2;\bm{\tau}_3)}&\textup{ by } \eqref{commute}\\
	&=\overline{(\bm{\Theta},\epsilon_r\bm u_h)_{\mathcal T_h}}&\textup{ by } \eqref{error_proof1}\\
	&=(\epsilon_r\bm u_h,\bm{\Theta})_{\mathcal T_h}\\
	&=( \epsilon_r\bm u_h,(\bm{\Theta}-\bm z_h-\nabla \xi_h))_{\mathcal T_h}
	+(\epsilon_r\bm u_h,\bm u_h^{\rm c}-\bm u_h)_{\mathcal T_h}
	+(\epsilon_r\bm u_h,\bm u_h)_{\mathcal T_h}&\textup{ by } \eqref{decom2}\\
	&=( \epsilon_r\bm u_h,\bm{\Theta}-\bm z_h)_{\mathcal T_h}
	+( \epsilon_r\bm u_h,(\bm u_h^{\rm c}-\bm u_h))_{\mathcal T_h}
	+( \epsilon_r\bm u_h,\bm u_h)_{\mathcal T_h}&\textup{ by } \eqref{Maxwell_equation_HDG_form_ori_3}.
	\end{align*}
	Then by the \Cref{vc}, \eqref{def_proof3}, the Cauchy-Schwarz inequality and  Young's inequality, we have 
	\begin{align}\label{error_proof4}
	\begin{split}
	{\rm{Re}}~[{\mathscr B_h^-(\bm{\sigma}_h;\bm{\tau}_2^{\star})}] &\ge \frac{1}{2}\|\sqrt{{\rm Re}~(\epsilon_r)}~\bm u_h\|_{\mathcal T_h}^2-
	Ch^2\|{\bf h}^{-\frac{1}{2}}\bm n\times[\![\bm u_h]\!]\|_{\mathcal F_h}^2\\
	& \ge \frac{1}{2}\|\sqrt{{\rm Re}~(\epsilon_r)}~\bm u_h\|_{\mathcal T_h}^2-
	Ch^2\vertiii{\bm{\sigma}_h}^2_{h}.
	\end{split}
	\end{align}
	
	Finally, we  take  $\bm{\tau}_h=\bm{\tau}_1+\frac{2C_4(1+\kappa^2)\|\epsilon_r\|_{L^\infty(\Omega)}}{\bar\epsilon_r}\bm{\tau}_2^\star$. {Then by}  \eqref{def_proof00} and \eqref{error_proof4} and letting $h$ be small enough we get the desired result.
\end{proof}

Our final lemma gives the desired error estimate:
\begin{lemma}
	Let $\mu_r$ and $\epsilon_r$ satisfy  Assumption \ref{Assumptions_coes}. {Suppose  $\kappa^2$ is not an eigenvalue of} problem \eqref{maxwell_eig}. Let $(\bm q, \bm u, p)\in \bm H({\rm curl}, \Omega)\times  \bm H_0({\rm curl}, \Omega)\times H_0^1(\Omega)$ and $\bm{\sigma}_h=(\bm q_h,\bm u_h, \widehat{\bm u}_h, p_h)\in \bm{\Sigma}_h  = \bm Q_h\times\bm U_h\times \widehat{\bm U}_h\times P_h$ be the solution of \eqref{con_pro} and \eqref{HDG_formulation}, respectively. Then for $h$ small enough, we have
	\begin{align*}
	&\|\bm q -\bm q_h\|_{\mathcal T_h} + \|\bm u -\bm u_h\|_{\mathcal T_h}+\|\nabla(p - p_h)\|_{\mathcal T_h} \\&\quad \le C\left(\|\bm{\Pi}_m^o \bm q-\bm q\|_{\mathcal T_h} +\|\bm{\mathcal J}_h^{\rm curl}\bm u-\bm u\|_{\mathcal T_h} + \|\bm{\mathcal J}_h^{\rm div} (\bm \nabla\times \bm u)-\bm \nabla\times\bm u\|_{\mathcal T_h}+\|\nabla (\mathcal I_hp - p)\|_{\mathcal T_h}\right),
	\end{align*}
	where $C$ depends  on  $\kappa$, $\Omega$, $\epsilon_r$ and $\mu_r$.
\end{lemma}

\begin{proof}  
	First, let $\bm \sigma = (\bm q, \bm u, \bm u, p)$ and $ \mathscr I_h \bm \sigma 	=(\bm{\Pi}_m^o\bm q,\bm{\mathcal J}_h^{\rm curl}\bm u,\bm n\times\mathscr\bm{\mathcal J}_h^{\rm curl} \bm u\times\bm n,\mathcal I_h p)$, where $\bm{\Pi}_m^o$ and  $\mathcal I_h$ denote the standard $L^2$ projection and the Scott-Zhang interpolation, respectively.   Then
	\begin{align*}
	\vertiii{\mathscr I_h \bm \sigma-\bm{\sigma}_h}_h &\le  	C\sup_{\bm 0\neq{\bm{\tau}}_h\in\bm{\Sigma}_h}\frac{{\rm Re}~[\mathscr B_h^-(\mathscr I_h \bm \sigma-\bm{\sigma}_h;\bm{\tau}_h)]}{\vertiii{\bm{\tau}_h}_h}&\textup{ by \eqref{dis_inf_sup_1}}\\
	&=C	\sup_{\bm 0\neq {\bm{\tau}}_h\in\bm{\Sigma}_h}\frac{{\rm Re}~[\mathscr B_h^-(\mathscr I_h \bm \sigma-\bm{\sigma};\bm{\tau}_h)]}{\vertiii{\bm{\tau}_h}_h}&\textup{ by \eqref{HDG_formulation1}}\\
	&\le C\left(\|\bm{\Pi}_m^o\bm q- \bm q\|_{\mathcal T_h}+\|\bm{\mathcal J}_h^{\rm curl}\bm u-\bm u\|_{\mathcal T_h}\right.\\
	&\quad + \left.\|\bm \nabla\times (\bm{\mathcal J}_h^{\rm curl}\bm u-\bm u)\|_{\mathcal T_h}+\|\nabla (\mathcal I_h p -p)\|_{\mathcal T_h}\right)&\textup{ by \eqref{boundnessofBminus}}\\
	&=C\left(\|\bm{\Pi}_m^o \bm q-\bm q\|_{\mathcal T_h} +\|\bm{\mathcal J}_h^{\rm curl}\bm u-\bm u\|_{\mathcal T_h}\right.\\
	&\quad + \left.\|\bm{\mathcal J}_h^{\rm div} (\bm \nabla\times \bm u)-\bm \nabla\times\bm u\|_{\mathcal T_h}+\|\nabla (\mathcal I_hp - p)\|_{\mathcal T_h}\right)&\textup{ by \eqref{curl_div_commute}}.
	\end{align*}
	By the definition of $\vertiii{\cdot}_h$ in  \eqref{trip_morm_def}  we have
	\begin{align*}
	&\|\bm{\Pi}_m^o\bm q -\bm q_h\|_{\mathcal T_h} + \|\bm{\mathcal J}_h^{\rm curl}\bm u -\bm u_h\|_{\mathcal T_h}+\|\nabla(\mathcal I_hp - p)\|_{\mathcal T_h}  \\&\quad \le C\left(\|\bm{\Pi}_m^o \bm q-\bm q\|_{\mathcal T_h} +\|\bm{\mathcal J}_h^{\rm curl}\bm u-\bm u\|_{\mathcal T_h} + \|\bm{\mathcal J}_h^{\rm div} (\bm \nabla\times \bm u)-\bm \nabla\times\bm u\|_{\mathcal T_h}+\|\nabla (\mathcal I_hp - p)\|_{\mathcal T_h}\right).
	\end{align*}
	Combined with the  triangle inequality  we get the desired result.
\end{proof}

\begin{remark} \label{remark1}
	We notice that the above estimates are optimal with respect to the regularity of $\bm u$ and $p$. If $\bm u$ and $p$ are   smooth enough, one would need to use a dual argument to get the optimal estimate for $\bm u-\bm u_h$ with respect to the degree of polynomials.
\end{remark}

\section{Numerical experiments}
\label{Numerical_experiments}

In this section, we present three numerical tests of the HDG-CG method  for  Maxwell's equations.  The domain of  the  following three examples is the unit cube $\Omega=(0, 1)\times (0,1 )\times (0,1)$.

\begin{example}\label{example1}
	We first test the convergence rate of the  method for  the Maxwell equations with smooth {(in fact constant) }coefficients with different  wave numbers $\kappa$. More specifically, the data is chosen to be
	\begin{gather*}
	\epsilon_r=1+2i, \ \  \mu_r=0.2-0.4i,  \ \ \bm a=\left(1,2\sqrt{3},2\right)^{\rm T}\\
	\bm d=\left(0,-\frac{1}{2},\frac{\sqrt{3}}{2}\right)^{\rm T},  \  \ \bm u=\bm a \exp(i\kappa{\bm x\cdot\bm d}), \   \ p=0, 
	\end{gather*}
	and the source term is  chosen to match the exact solution of  \Cref{Maxwell_PDE_ori}. The approximation errors  are listed in \Cref{tab2,tab3}.  { In this case the domain is convex and the coefficients are smooth, so that a standard duality 
		approach could be used to prove convergences. The results in \Cref{tab2} are for polynomial degree $k=1,2,3$ and show that optimal order convergence is seen when the solution is smooth and $\kappa>0$.  The same is seen even when
		$\kappa=0$ showing that the HDG-CG method can be used for low frequency (even zero frequency) problem and that a CG Lagrange multiplier space
		is sufficient to achieve stability even when $\kappa=0$.}
	\begin{table}[H]
		\caption{History of convergence for  \Cref{example1} with  $\kappa=1$} \label{tab2}
		\centering
		\begin{tabular}{c|c|c|c|c|c}
			\Xhline{1pt}

			\multirow{2}{*}{$k$}
			&\multirow{2}{*}{$\frac{h}{\sqrt{3}}$}	
			&\multicolumn{2}{c|}{$\|\bm q-\bm q_h\|_{\bm L^2(\Omega)}/\|\bm q\|_{\bm L^2(\Omega)}$}	
			&\multicolumn{2}{c}{$\|\bm u-\bm u_h\|_{\bm L^2(\Omega)}/\|\bm u\|_{\bm L^2(\Omega)}$}		\\
			\cline{3-6}
			& &Error &Rate
			&Error &Rate
			\\
			
			\cline{1-6}
			\multirow{4}{*}{  $1$}
			&	1/2	&	1.23E-01	&		&	1.55E-01	&		\\
			&	1/4	&	6.38E-02	&	0.95 	&	4.37E-02	&	1.82\\
			&	1/8	&	3.23E-02	&	0.98 	&	1.15E-02	&	1.93\\
			&	1/6	&	1.62E-02	&	0.99 	&	2.96E-03	&	1.96\\

			\cline{1-6}
			\multirow{4}{*}{  $2$}
			&	1/2&	2.86E-02	&		&	1.06E-02	&	\\
			&	1/4&	7.50E-03	&	1.93 	&	1.33E-03	&	3.00 \\
			&	1/8&	1.92E-03	&	1.97 	&	1.67E-04	&	2.99 \\
			&	1/12	&	8.58E-04	&	1.99 	&	4.95E-05	&	3.00\\

			\cline{1-6}
			\multirow{4}{*}{  $3$}
			&	1/2&	1.23E-03	&		&	4.81E-04	&	\\
			&	1/4&	1.60E-04	&	2.95 	&	3.29E-05	&	3.87 \\
			&	1/8&	2.03E-05	&	2.98 	&	2.17E-06	&	3.92 \\

			\Xhline{1pt}

		\end{tabular}
	\end{table}
	
	\begin{table}[H]
		\caption{History of convergence for  \Cref{example1} with $\kappa=0$} \label{tab3}
		\centering
		\begin{tabular}{c|c|c|c|c|c}
			\Xhline{1pt}
			\multirow{2}{*}{$k$}
			&\multirow{2}{*}{$\frac{h}{\sqrt{3}}$}	
			&\multicolumn{2}{c|}{$\|\bm q-\bm q_h\|_{\bm L^2(\Omega)}/\|\bm q\|_{\bm L^2(\Omega)}$}	
			&\multicolumn{2}{c}{$\|\bm u-\bm u_h\|_{\bm L^2(\Omega)}/\|\bm u\|_{\bm L^2(\Omega)}$}		\\
			\cline{3-6}
			& &Error &Rate
			&Error &Rate
			\\
			
			\cline{1-6}
			\multirow{4}{*}{  $1$}
			& {1/}2	&	1.20E-01	&		&	1.37E-01	&	\\
			&	 {1/}4	&	6.34E-02	&	0.92 	&	4.09E-02	&	1.74 \\
			&	 {1/}8	&	3.23E-02	&	0.97 	&	1.11E-02	&	1.88 \\
			&	 {1/}12	&	2.16E-02	&	0.99 	&	5.09E-03	&	1.93\\ 
			
			\cline{1-6}
			\multirow{4}{*}{  $2$}
			&	 {1/}2	&	7.24E-03	&		&	5.52E-03	&	\\
			&	 {1/}4	&	1.93E-03	&	1.91 	&	7.13E-04	&	2.95 \\
			&	 {1/}8	&	4.98E-04	&	1.95 	&	9.00E-05	&	2.99 \\
			
			\Xhline{1pt}
		\end{tabular}
	\end{table}
\end{example}

\begin{example}\label{example2}
	We  next test the convergence rate of the  method for  Maxwell's equations with wave number $\kappa =1$ when  the  coefficients are piecewise smooth. More specifically, the data is chosen as
	\begin{gather*}
	\mu_r=\begin{cases}
	0.2-0.4i, & x<0.5,\\
	0.25-0.25i, & x\ge 0.5,
	\end{cases}
	\qquad  \epsilon_r=\begin{cases}
	1+2i, & x<0.5,\\
	2+2i, & x\ge 0.5,
	\end{cases} \\
	\bm u =[u_1, u_2, u_3]^{\rm T}, \ u_1=1, \  u_2=(x-0.5)^2z, \ u_3=(x-0.5)^2y,\ \  p=0.
	\end{gather*}
	The source term is  chosen to match the exact solution of  \Cref{Maxwell_PDE_ori} and the approximation errors  are listed in \Cref{tab4}.   {In this case the duality approach would need to handle discontinuous coefficients which limit the
		regularity of the dual solution to $H^s$ with $s</12$.  Our analysis covers this case.  Since the chosen true solution
		is smooth, we expect optimal order convergence as is seen in \Cref{tab4}.}
	
	\begin{table}[H]
		\caption{History of convergence for  \Cref{example2}.} \label{tab4}
		\centering
		\begin{tabular}{c|c|c|c|c|c}
			\Xhline{1pt}
			\multirow{2}{*}{$k$}
			&\multirow{2}{*}{$\frac{h}{\sqrt{3}}$}	
			&\multicolumn{2}{c|}{$\|\bm r-\bm r_h\|_{\bm L^2(\Omega)}/\|\bm r\|_{\bm L^2(\Omega)}$}	
			&\multicolumn{2}{c}{$\|\bm u-\bm u_h\|_{\bm L^2(\Omega)}/\|\bm u\|_{\bm L^2(\Omega)}$}		\\
			\cline{3-6}
			& &Error &Rate
			&Error &Rate
			\\
			
			\cline{1-6}
			\multirow{5}{*}{  $1$}
			&	 {1/}2	&	5.10E-01	&		&	3.05E-01	&	\\
			&	 {1/}4	&	2.75E-01	&	0.89 	&	8.74E-02	&	1.80 \\
			&	 {1/}8	&	1.42E-01	&	0.95 	&	2.32E-02	&	1.91 \\
			&	 {1/}16	&	7.22E-02	&	0.98 	&	6.02E-03	&	1.95 \\
			&	 {1/}20	&	5.79E-02	&	0.99 	&	3.88E-03	&	1.96 \\
			
			\cline{1-6}
			\multirow{4}{*}{  $2$}
			&	 {1/}2	&	8.05E-02	&		&	2.45E-02	&	\\
			&	 {1/}4	&	2.12E-02	&	1.93 	&	3.05E-03	&	3.00 \\
			&	 {1/}8	&	5.42E-03	&	1.96 	&	3.84E-04	&	2.99 \\
			&	 {1/}12	&	2.43E-03	&	1.98 	&	1.14E-04	&	2.99 \\
			
			\Xhline{1pt}
		\end{tabular}
	\end{table}
\end{example}

\begin{example}\label{example3}
	Finally,  we test the convergence rate of the  method for  the Maxwell equations with wave number $\kappa =1$, when both the  coefficients and the exact solution are piecewise smooth. More specifically, the data is chosen as
	\begin{gather*}
	\mu_r=\begin{cases}
	0.2-0.4i, & x<0.5,\\
	0.25-0.25i, & x\ge 0.5,
	\end{cases}
	\qquad  \epsilon_r=\begin{cases}
	1+2i, & x<0.5,\\
	2+2i, & x\ge 0.5,
	\end{cases} \\
	\bm u =[u_1, u_2, u_3]^{\rm T}, \ u_1=\begin{cases}
	2 & x<0.5,\\
	1, & x\ge 0.5,
	\end{cases},
	\  u_2=(x-0.5)^2z, \ 
	u_3= (x-0.5)^2y, \ \  p=0.
	\end{gather*}
	and the source term is  chosen to match the exact solution of  \Cref{Maxwell_PDE_ori} and the approximation errors  are listed in \Cref{tab5}.   {In this case the solution is piecewise analytic and the mesh is chosen so that the surface of
		discontinuity $x=0.5$ is a union of faces in the mesh. Thus standard error estimates for polynmial interpolation applied tetrahedron by tetrahedron give an optimal error estimate.  This is confirmed in \Cref{tab5}.}
	\begin{table}[H]
		\caption{History of convergence for  \Cref{example3}.} \label{tab5}
		\centering
		\begin{tabular}{c|c|c|c|c|c}
			\Xhline{1pt}

			\multirow{2}{*}{$k$}
			&\multirow{2}{*}{$\frac{h}{\sqrt{3}}$}	
			&\multicolumn{2}{c|}{$\|\bm q-\bm q_h\|_{\bm L^2(\Omega)}/\|\bm q\|_{\bm L^2(\Omega)}$}	
			&\multicolumn{2}{c}{$\|\bm u-\bm u_h\|_{\bm L^2(\Omega)}/\|\bm u\|_{\bm L^2(\Omega)}$}		\\
			\cline{3-6}
			& &Error &Rate
			&Error &Rate
			\\
			
			\cline{1-6}
			\multirow{5}{*}{  $1$}
			&	2	&	5.10E-01	&		    &	1.53E-01	&		\\
			&	4	&	2.75E-01	&	0.89 	&	4.38E-02	&	1.80 	\\
			&	8	&	1.42E-01	&	0.95 	&	1.16E-02	&	1.91 	\\
			&	16	&	7.22E-02	&	0.98 	&	3.02E-03	&	1.95 \\
			&	20	&	5.79E-02	&	0.99 	&	1.95E-03	&	1.96 \\

			\cline{1-6}
			\multirow{4}{*}{  $2$}
			&	2	&	8.05E-02	&		&	1.23E-02	&	\\
			&	4	&	2.12E-02	&	1.93 	&	1.53E-03	&	3.00 \\
			&	8	&	5.42E-03	&	1.96 	&	1.93E-04	&	2.99 \\
			&	12	&	2.43E-03	&	1.98 	&	5.73E-05	&	2.99 \\
			
			\Xhline{1pt}
		\end{tabular}
	\end{table}
	
\end{example}

\section{Conclusion}
{We have proved that the HDG-CG method for the time harmonic Maxwell system converges even in the presence of general piecewise smooth coefficients, as are usually encountered in practical applications.  Our numerical results suggest that the method is stable even when $\kappa=0$ so that the use of expanded HDG spaces for the Lagrange multiplier $p$ is not needed, and a CG space is sufficient.}

{We expect that the method of proof given in our paper will be useful for other HDG methods which are intended for use 
	on heterogeneous media. The dependence of the coefficients in the estimates on the wave number $\kappa$ was not traced, and this should be done in the future.  However the simple model problem used here would need to be revised
	to have Robin type boundary conditions (or other boundary conditiond conditions in which the dependency of the solution of
	continuous problem on $\kappa$ is known).}

\bibliographystyle{siamplain}
\bibliography{complete}

\end{document}